\documentclass{amsart}

\usepackage{amsfonts}
\usepackage{latexsym}
\usepackage{amscd,amsmath,amsthm,amssymb}

\newcommand{\C}{\mathbb C}
\newcommand{\R}{\mathbb R}

\newcommand{\HH}{{\mathbb{H}}}

\def\a'{\`a}
\def\e'{\`e}
\def\o'{\`o}
\def\u'{\`u}

\newtheorem{teor}{Theorem}[section]
\newtheorem{coro}[teor]{Corollary}
\newtheorem{lemma}[teor]{Lemma}
\newtheorem{prop}[teor]{Proposition}

\newtheorem{defi}[teor]{Definition}

\theoremstyle{definition}
\newtheorem*{Acknow}{Acknowledgments}

\begin{document}

\def\pqK{para-quaternionic K\"ahler \,}
\def\aeH{almost $\epsilon$-Hermitian \,}

 \title{(Para-)Hermitian and (para-)K\"ahler Submanifolds of a para-quaternionic K\"ahler manifold}

\author{Massimo Vaccaro}
\address{Dipartimento
dell'Ingegneria di Informazione e Matematica Applicata,
Universit\a' di Salerno, 84084 - Fisciano (SA) , Italy}
\email{massimo\_vaccaro@libero.it}
\thanks{Work done under the programs of GNSAGA-INDAM of C.N.R. and PRIN07 "Riemannian metrics and differentiable structures" of MIUR (italy)}
\date{\today} 

\maketitle

\markboth{ \rm{MASSIMO VACCARO}} {\rm{(PARA-)HERMITIAN AND (PARA-)K\"AHLER SUBMANIFOLD OF A PQK MANIFOLD}}

\textbf{Abstract}: On a para-quaternionic K\"ahler manifold
 $(\widetilde M^{4n},Q,\widetilde g)$, which is first of all a pseudo-Riemannian manifold, a natural definition of
(almost) K\"ahler and (almost) para-K\"ahler submanifold $(M^{2m},\mathcal{J},g)$ can be given where $\mathcal{J}=J_1|_M$
 is a (para-)complex structure on $M$ which is the  restriction of a section $J_1$ of the para-quaternionic bundle $Q$.
In this paper, we extend to such a submanifold $M$  most of the results proved by Alekseevsky and Marchiafava, 2001, where  Hermitian and K\"ahler submanifolds of a quaternionic K\"ahler manifold have been studied.

Conditions for the integrability of an almost (para-)Hermitian structure on $M$ are given.  Assuming that the scalar curvature  of $\widetilde M$ is non zero, we  show that any almost (para-)K\"ahler submanifold is (para-)K\"ahler and
 moreover that  $M$ is (para-)K\"ahler iff it is totally (para-)complex.
Considering   totally (para-)complex submanifolds of maximal dimension $2n$, we identify the second fundamental form $h$ of $M$ with a tensor
$C= J_2 \circ h \in TM\otimes S^2T^*M$ where $J_2 \in Q$ is a compatible para-complex structure anticommuting with $J_1$. This tensor, at any point
$x\in M$, belongs to the first prolongation
$S^{(1)}_\mathcal{J}$ of the space
$S_\mathcal{J} \subset \operatorname{End}T_xM$ of
symmetric endomorphisms anticommuting with $\mathcal{J}$. When $\widetilde M^{4n}$ is a symmetric manifold
 the condition for  a (para-)K\"ahler submanifold $M^{2n}$ to be  locally symmetric is given.
In the case when  $\widetilde M$ is a para-quaternionic space form, it is shown, by using Gauss and Ricci equations, that a (para-)K\"ahler
submanifold $M^{2n}$ is curvature invariant. Moreover it is a locally symmetric Hermitian submanifold
iff  the  $\mathfrak u(n)$-valued 2-form $[C,C]$ is parallel.
Finally  a characterization of  \textit{parallel} K\"ahler  and para-K\"ahler submanifold
of maximal dimension is given.

\section{Introduction}

A pseudo-Riemannian manifold $(M^{4n},g)$ with the holonomy group contained in $Sp_1(\R) \cdot Sp_{n}(\R)$ is called a para-quaternionic
K\"ahler manifold. This means that
there exists a 3-dimensional parallel subbundle $Q \subset End TM$
 of the bundle of endomorphisms which is locally
generated by three skew-symmetric anticommuting endomorphisms
$I,J,K$ satisfying the following \textit{para-quaternionic relations}
\[
-I^2=J^2=K^2=Id , \; IJ=-JI=K.
\]
The subbundle ${Q} \subset End(TM)$ is called a
para-quaternionic structure. Any para-quaternionic K\"ahler manifold
is an Einstein  manifold \cite{AC2}.



 Let $\epsilon= \pm 1$;  a submanifold $(M^{2m},\mathcal{J^\epsilon}= {J^\epsilon}|_{TM},g)$ of the \pqK manifold $(\widetilde M^{4n},Q, \widetilde g)$,
  where $M\subset \widetilde M$ is a submanifold, the induced metric $g=\widetilde g|_M$ is non-degenerate,
and  $J^\epsilon$ is   a section of the bundle $Q_{|M}  \rightarrow  M$ such that
${J^\epsilon}TM=TM, \; \ ({J^\epsilon})^2=\epsilon Id$,
  is called an {\it \aeH submanifold}.


An \aeH submanifold $(M^{2m},\mathcal{J^\epsilon},g)$ of a \pqK manifold $(\widetilde M^{4n},Q, \widetilde g)$
is called
{\it $\epsilon$-Hermitian} if the almost $\epsilon$-complex
structure $\mathcal{J^\epsilon}$ is integrable, {\it almost
$\epsilon$-K\"ahler} if the K\"ahler form $F=g\circ
\mathcal{J^\epsilon}$ is closed and {\it $\epsilon$-K\"ahler} if
$F$ is parallel. Note that  $\epsilon$-K\"ahler submanifolds are minimal (\cite{AC1}).

We will always assume that $\widetilde M^{4n}$ {\it has non zero reduced scalar
curvature $\nu=scal/(4n(n+2))$}.

In  section 3 we study an \aeH submanifold
$(M^{2m},\mathcal{J^\epsilon},g)$ of the \pqK manifold $\widetilde
M^{4n}$ and give the necessary and sufficient condition  to
be $\epsilon$-Hermitian.
If  furthermore $M$ is
analytic, we show that a sufficient condition for integrability is that $\textrm{codim}\, \overline{T_xM}>2$ at some point $x \in M$ where
by $\overline{T_xM}$ we denote the maximal $Q_x$-invariant subspace of $T_xM$. Then, as an application,
we prove that, if  the set $U$ of points $x\in M$
where the Nijenhuis tensor of $\mathcal{J^\epsilon}$ of an \aeH
submanifold of dimension $4k$ is not zero is open and dense in $M$
and $\overline{T_xM}$ is non degenerate, then
$M$ is a  para-quaternionic submanifold.

In fact, by extending a classical result  of quaternionic geometry (see \cite{A}, \cite{G}), we show that a non degenerate
para-quaternionic submanifold of a para-quaternionic K\"ahler
manifold is  totally geodesic, hence a para-quaternionic K\"ahler
submanifold.

In section 4, we give two equivalent necessary and sufficient conditions for an
\aeH manifold to be $\epsilon$-K\"ahler. We prove that an almost
$\epsilon$-K\"ahler submanifold $M^{2m}\,$ of a
\pqK manifold $\widetilde M^{4n}$ is $\epsilon$-K\"ahler and,
hence, a minimal submanifold (see \cite{AC1}) and give some local characterizations
of such a submanifold (Theorem \ref{almost Hermitian to be Kaehler}).
In Theorem \ref{second fundamental form of a Kaehler submanifold} we prove
that the second fundamental form $h$ of  a  $\epsilon$-K\"ahler  sub\-ma\-ni\-fold  $M$
satisfies the fundamental identity
\[
h(\mathcal{J^\epsilon} X,Y)=J^\epsilon
h(X,Y) \qquad \qquad \forall \, X,Y \in TM
\]
and that, conversely, if the above identity  holds on an \aeH sub\-ma\-ni\-fold
$M^{2m}$ of $\widetilde M^{4n}$  then $M^{2m}$ is either a
$\epsilon$-K\"ahler submanifold or a para-quaternionic (K\"ahler)
submanifold and these cases cannot happen simultaneously. In
particular,  we prove that an \aeH submanifold $M$ is
$\epsilon$-K\"ahler if and only if it is {\it totally
$\epsilon$-complex}, i.e. it satisfies the condition
$J_2T_xM \, \bot \, T_xM \quad \forall \, x\in M
$, where $J_2 \in Q$ is a compatible para-complex structure anticommuting with $J^\epsilon$.

In section 5, we study an $\epsilon$-K\"ahler submanifold $M$  of
maximal dimension $2n$ in a \pqK manifold $(\widetilde
M^{4n},Q,\widetilde g)$ (still assuming $\nu \neq 0$). Using the field of isomorphisms $J_2: TM
\rightarrow T^\bot M$ between the tangent and the normal bundle,
we identify, as in \cite{AM}, the second fundamental form $h$ of $M$ with a tensor
$C=J_2\circ h \in TM\otimes S^2T^*M$. This tensor, at any point
$x\in M$, belongs to the first prolongation
$S^{(1)}_\mathcal{J^\epsilon}$ of the space
$S_\mathcal{J^\epsilon} \subset \operatorname{End}T_xM$ of
symmetric endomorphisms anticommuting with $\mathcal{J^\epsilon}$.
Using the  tensor $C$, we present the Gauss-Codazzi-Ricci e\-qua\-tions in a simple form and derive from it
 the necessary and sufficient conditions for
the $\epsilon$-K\"ahler submanifold $M$  to be parallel and to be
curvature invariant (i.e. $ \widetilde R_{XY}Z \in TM, \;
\forall \, X,Y,Z \in TM $). In subsection 5.4 we study  a maximal
$\epsilon$-K\"ahler submanifold $M$ of a (locally) symmetric \pqK
space $\widetilde M^{4n}$ and  get the necessary and sufficient
conditions for $M$ to be a locally symmetric manifold in
terms of the tensor $C$. In particular, if $\widetilde M^{4n}$ is a quaternionic space form, then the
$\epsilon$-K\"ahler submanifold $M$ is curvature invariant. In this case,
$M$ is symmetric if and only if the  2-form
$$[C,C]: X\wedge Y \mapsto [C_X,C_Y] \qquad  \quad X,Y
\in TM,
$$
with values in the unitary algebra of the $\epsilon$-Hermitian structure and that satisfies the first and the second Bianchi identity, is
parallel. Moreover  $M$ is a
totally $\epsilon$-complex totally geodesic submanifold of
 the  quaternionic space form $\widetilde M^{4n}$
if and only if
\[\text{\rm Ric}_M=\frac{\nu}{2}(n+1)g\]
(see Proposition \ref{Ricci curvature of a Kaehler submanifold of a paraquaternionic space form}).

In Section 6 we  characterize a maximal $\epsilon$-K\"ahler submanifold $M$ of the \pqK manifold
$\widetilde M^{4n}$ with parallel non zero
 second fundamental form $h$, or shortly, {\it parallel $\epsilon$-K\"ahler
 submanifold}. In terms of the tensor $C$, this means that
$$
\nabla_XC = -\epsilon \omega(X)\mathcal{J^\epsilon} \circ C, \qquad
\qquad X \in TM
$$
where $\omega= \omega_1 |_{TM}$  and $\nabla$ is the Levi-Civita
connection of $M$. When  $(M^{2n},\mathcal J,g)$, where $ \mathcal J=\mathcal J^\epsilon, \epsilon=-1$, is a parallel not totally
geodesic K\"ahler submanifold, the covariant tensor $g\circ C$ has the form $gC=q+\overline q$
where $q \in S^3(T_x^{*1,0}M)$ (resp. $\bar q \in S^3(T_x^{*0,1}M)$) is a holomorphic (resp. antiholomorphic) cubic form. We prove
that any parallel, not totally geodesic, K\"ahler submanifold $(M^{2n},\mathcal J,g)$
of a \pqK manifold $(\widetilde{M}^{4n}, Q, g)$ with $\nu \neq 0$
 admits a pair of parallel holomorphic line subbundle
$L=\text{span}_{\mathbb C}(q)$ of the bundle $S^3T^{*1,0}M$
and $\overline{L}=\text{span}_{\mathbb C}(\overline{q})$ of the bundle $S^3T^{*0,1}M$
such that the connection induced on $L$ (resp. $\overline{L}$) has the curvature $R^L=-i\nu
g\circ \mathcal J=-i\nu F$ (resp. $R^{\overline{L}}=i\nu
g\circ \mathcal J=i\nu F$). In case  $(M^{2n},\mathcal J,g)$ where $ \mathcal J=\mathcal J^\epsilon, \epsilon=+1$, is a parallel not totally
geodesic para-K\"ahler submanifold of $(\widetilde{M}^{4n}, Q, \widetilde g)$
we have
$
gC=q^{+} + q^{-} \in S^3(T^{*+}M) + S^3(T^{*-}M) \,
$ where $TM=T^{+} + T^{-}$ is the bi-Lagrangean decomposition of the tangent bundle.
We prove that, in this case,  the pair of real line subbundle
$L^+:=\R q^+ \subset S^3(T^{*+}M) $ and $L^-:=\R q^- \subset
S^3(T^{*-}M) )$ are globally defined on $M$ and parallel w.r.t the
Levi-Civita connection which defines a connection $\nabla^{L^+}$ on $L^+$  (resp.
$\nabla^{L^-}$ on $L^-$) whose curvature is
\[
\begin{array}{l} 
R^{L^+}= \nu F, \quad (\text{resp.} \quad  R^{L^-}= -\nu F).
\end{array}
\]

\section{Para-quaternionic K\"ahler manifolds} \label{section para-quaternionic Kaehler manifolds}

For a more detailed study of para-quaternionic K\"ahler manifolds
see \cite{39}, \cite{AC1}, \cite{LD}, \cite{DJS}, \cite{IMV}. Moreover for a survey on para-complex geometry see \cite{AMT}, \cite{30}.
\begin{defi} {\rm{(\cite{AC1})}}
Let $(\epsilon_1, \epsilon_2, \epsilon_3)=(-1,1,1)$ or a
permutation thereof. An \textbf{almost para-quaternionic
structure} on a differentiable manifold $\widetilde
 M$ (of di\-men\-sion $2m$) is a
rank 3 subbundle $Q \subset End T \widetilde
M$, which is locally generated by
three anticommuting fields of endomorphism $J_1, J_2, J_3=J_1
J_2$, such that  $J_\alpha^2=\epsilon_\alpha Id$. Such a triple
will be called a \textbf{standard basis} of $Q$. A linear
connection $\widetilde
 \nabla$ which preserves $Q$
 is called an \textbf{almost para-quaternionic
connection}. An almost para-quaternionic structure $Q$ is called a
\textbf{para-quaternionic structure} if $ \widetilde
M$ admits a
para-quaternionic connection i. e. a \textit{torsion-free}
connection which preserves $Q$. An \textbf{(almost)
para-quaternionic manifold} is a manifold endowed with an (almost)
para-quaternionic structure.
\end{defi}

Observe that $J_\alpha J_\beta=\epsilon_3 \epsilon \gamma
J_\gamma$ where $(\alpha,\beta,\gamma)$ is a cyclic permutation of
(1,2,3).

\begin{defi} {\rm{(\cite{AC1})}}
An \textbf{(almost) para-quaternionic Hermitian manifold}
$(\widetilde M,Q,\widetilde
 g)$ is a pseudo-Riemannian manifold $(\widetilde
 M,\widetilde
 g)$ endowed
with an (almost) para-quaternionic structure $Q$ consisting of
skew-symmetric endomorphisms. The non degeneracy of the metric
implies that $\dim \widetilde
 M=4n$ and the signature of $\widetilde
 g$ is neutral.  $(\widetilde
 M^{4n},Q,\widetilde
 g)$, $n>1$, is called a
\textbf{para-quaternionic K\"ahler manifold} if the Levi-Civita
connection preserves $Q$.
\end{defi}

\begin{prop} \label{curvature in pqKm} {\rm{(\cite{AC2})}} The curvature tensor $\widetilde
 R$ of a
para-quaternionic K\"ahler ma\-ni\-fold $(\widetilde
 M,Q,\widetilde
 g)$, of dimension
$4n>4$, at any point admits a decomposition
\begin{equation} \label{curvature tensor of pqKm}
\widetilde  R= \nu  R_0 + W,
\end{equation}
where $\nu= \frac{scal}{4n(n+2)}$ is the reduced scalar curvature,
\begin{equation} \label{curvature tensor of the para-quaternionic projective space}
R_0(X,Y):=\frac{1}{2} \sum_\alpha \epsilon_\alpha \widetilde g(J_\alpha X,Y)
J_\alpha + \frac{1}{4}( X \wedge Y -\sum_\alpha \epsilon_\alpha
J_\alpha X \wedge J_\alpha Y), \quad X,Y \in T_pM,
\end{equation}
is the curvature tensor of the para-quaternionic projective space
of the same di\-men\-sion as $\widetilde M$ and $W$ is a trace-free $Q$-invariant
algebraic curvature tensor, where $Q$ acts by derivations. In
particular, $\widetilde R$ is $Q$-invariant.
\end{prop}

We define a \textbf{para-quaternionic K\"ahler manifold of
dimension 4} as a pseudo-Riemannian  manifold endowed with a
parallel skew-symmetric para-qua\-ter\-nio\-nic K\"ahler  structure
whose curvature tensor admits the decomposition (\ref{curvature
tensor of pqKm}).

Since the Levi-Civita connections $\widetilde  \nabla$ of a para-quaternionic
K\"ahler manifold pre\-ser\-ves the para-quaternionic K\"ahler
structure $Q$, one can write
\begin{equation} \label{nabla J}
\widetilde \nabla J_\alpha= -\epsilon _\beta  \omega_\gamma \otimes J_\beta +
\epsilon _\gamma  \omega_\beta \otimes J_\gamma,
\end{equation}
where the $\omega_\alpha, \; \alpha=1,2,3$ are locally defined
1-forms and  $(\alpha,\beta,\gamma)$ is a cyclic permutation of
(1,2,3). We shall denote by $F_\alpha:=\widetilde g(J_\alpha \cdot,\cdot)$
the \textbf{K\"ahler form} associated with $J_\alpha$ and put
$F_\alpha' :=-\epsilon_\alpha F_\alpha$.

We recall the expression for the action of the curvature operator $\widetilde R(X,Y) , \; X,Y
\in T \widetilde M$ of $\widetilde M$, on $J_\alpha$:
\begin{equation}\label{action of the curvature operator}
\begin{array}{l}
[\widetilde R(X,Y),J_\alpha]  =
\epsilon_3 \nu(-\epsilon_\beta F_\gamma' (X,Y) J_\beta +
\epsilon_\gamma F_\beta' (X,Y) J_\gamma)
\end{array}
\end{equation}
where $(\alpha,\beta,\gamma)$ is a cyclic permutation of (1,2,3).

\begin{prop} {\rm{(\cite{AC1})}}
The locally defined K\"ahler forms satisfy the following
struc\-tu\-re equations
\begin{equation} \label{structure equations in pqKm}
\nu F_\alpha' := -\epsilon_\alpha \nu F_\alpha= \epsilon_3 ( d
\omega_\alpha - \epsilon_\alpha \omega_\beta  \wedge
\omega_\gamma),
\end{equation}
where $(\alpha,\beta,\gamma)$ is a cyclic permutation of (1,2,3).
\end{prop}


 By taking the exterior derivative of (\ref{structure equations in
 pqKm})
we get
\[
\begin{array}{ll}
\nu d F_\alpha' & = \epsilon_3 d ( d \omega_\alpha -
\epsilon_\alpha \omega_\beta  \wedge \omega_\gamma) = -\epsilon_3  ( \epsilon_\alpha d \omega_\beta  \wedge
\omega_\gamma -
\epsilon_\alpha \omega_\beta  \wedge d \omega_\gamma)\\
\end{array}.
\]
Since
$ d \omega_\beta= \epsilon_3 \nu F'_\beta+ \epsilon_\beta
\omega_\gamma \wedge \omega_\alpha$ and  $ d \omega_\gamma= \epsilon_3 \nu F'_\gamma+ \epsilon_\gamma
\omega_\alpha \wedge \omega_\beta$, we get
\[
\nu d F_\alpha'= -\epsilon_3 [ (\epsilon_\alpha \epsilon_3 \nu
F'_\beta \wedge \omega_\gamma)- (\epsilon_\alpha \omega_\beta
\wedge \epsilon_3 \nu F'_\gamma,)]
\]
that is   $\nu [d F'_\alpha -\epsilon_\alpha
(-F'_\beta \wedge \omega_\gamma + \omega_\beta \wedge
F'_\gamma)]=0$. Hence we have the following result.
\begin{prop}
On a para-quaternionic K\"ahler manifold the following integrability conditions hold
\begin{equation} \label{exterior differential of fundamental 2-form}
 \nu [d F'_\alpha -\epsilon_\alpha
(-F'_\beta \wedge \omega_\gamma + \omega_\beta \wedge
F'_\gamma)]=0, \qquad (\alpha,\beta,\gamma)=\text{\rm
cycl}(1,2,3).
\end{equation}
\end{prop}

\section{Almost $\epsilon$-Hermitian submanifolds of $\widetilde M^{4n}$}
The definition of an (almost) complex structure on
a differentiable manifold and the condition for its integrability are well known.
We just recall the following other definitions (see \cite{AC1}).
\begin{defi}
An \textbf{(almost)para-complex structure} on a
differentiable ma\-ni\-fold $M$ is a field of endomorphisms $J \in End
TM$ such that $J^2=Id$ and the $\pm 1$-eigenspace distributions
$T^{\pm}M$ of $J$ have the same rank. An almost para-complex
structure is called \textbf{integrable}, or \textbf{para-complex
structure}, if the distributions $T^{\pm}M$ are integrable or,
equivalently, the Nijenhuis tensor $N_J$, defined by
\[
N_J(X,Y)= [JX,JY]-J[JX,Y] -J[X,JY]+[X,Y], \qquad X,Y \in TM
\]
vanishes. An \textbf{(almost)para-complex manifold $(M,J)$} is a
manifold $M$ endowed with an (almost) para-complex structure.
\end{defi}

\begin{defi}
An \textbf{(almost) $\epsilon$-complex structure} $\epsilon \in \{-1,1\}$
on a dif\-fe\-ren\-tia\-ble manifold $M$ of dimension $2n$ is a field of
endomorphisms $J \in End TM$ such that $J^2= \epsilon Id$ and
moreover,  for $\epsilon =+1$  the eigendistributions $T^{\pm}M$
are of rank $n$.
An \textbf{$\epsilon$-complex manifold} is a differentiable
manifold endowed with an integrable (i.e. $N_J=0$)
$\epsilon$-complex structure.
\end{defi}

Consequently,  the notation (almost) $\epsilon$-Hermitian
structure, (almost) $\epsilon$-K\"ahler structure, etc.. will be
used with the same convention.

Let recall that  a submanifold of a pseudo-Riemannian manifold  is non degenerate if it has non degenerate  tangent spaces.

\begin{defi}
Let  $(\widetilde M^{4n},Q,\widetilde g)$  be a \pqK manifold.
A $\widetilde g$-non degenerate submanifold $M^{2m}$ of $\widetilde M$ is called
 an \textbf{\aeH
submanifold} of $\widetilde M$ if
there exists a section $J^\epsilon: M \rightarrow Q_{|M}$ such that
\[{J^\epsilon}TM=TM \qquad  (J^\epsilon)^2= \epsilon Id. \]
We will denote such submanifold  $(M^{2m},\mathcal{J}^{\epsilon},g)$ where
$(g= \widetilde g|_M, \;  \mathcal{J^\epsilon}= J^\epsilon|_M)$.
\end{defi}

For a classification of almost (resp. para-)Hermitian manifolds see \cite{24}, (resp. \cite{25},\cite{28}).

Notice (see \cite{MV},\cite{MV1},\cite{MV2}) that in any point $x \in
M$ the induced metric $g_x=<,>_x$ of an (almost) Hermitian submanifold
 has signature $2p,2q$ with $p+q=m$
whereas  the signature of  the metric    of an (almost)
para-Hermitian
 submanifold  is always neutral $(m,m)$. In both cases then the  induced metric is
\textbf{pseudo}-Riemannian  (and Hermitian).
Keeping in mind this fact,
we will not use the suffix "pseudo" in the following.

For any point $x\in M^{2m}$, we can always include $J^\epsilon$ into a
local frame $(J_1=J^\epsilon, J_2,J_3=J_1 J_2=-J_2 J_1)$ of
$Q$  defined in a neighbourhood $\widetilde U$ of $x$ in
$\widetilde M$ such that $J_2^2=Id$. Such frame will be called
\textbf{adapted} to the submanifold $M$ and in fact, since our considerations are local,
we will assume for simplicity that $\widetilde U \supset M^{2m}$ and
put
\[ F={F_1}_{|M}= g \circ \mathcal J^\epsilon, \quad  \quad \omega={\omega_1}_{|M} \, .\]
Moreover, we have
\begin{equation}
\widetilde \nabla J^\epsilon= - \omega_3 \otimes J_2 - \epsilon \, \omega_2
\otimes J_3
\end{equation}
where $\widetilde \nabla$ indicates the Levi-Civita connection on $\widetilde M$,
 and in complex case ($\epsilon=-1$), from
$(\epsilon_1,\epsilon_2,\epsilon_3)=(-1,1,1)$, we have $J_2
J_3=-J_1, \; J_3 J_1=J_2$ whereas in para-complex case, where
$(\epsilon_1,\epsilon_2,\epsilon_3)=(1,1,-1)$, we have $J_2
J_3=-J_1, \; J_3 J_1=-J_2$.

For any $x \in M$ we denote $\overline {T_x M}$ the maximal
para-quaternionic ($Q$-invariant) subspace of the tangent space
$T_xM$. Note that if
$(J_1,J_2,J_3)$ is an adapted basis in a point $x\in M$ then
$\overline{T_xM}=T_xM\cap J_2T_xM$.

We allow  $\overline {T_x M}$ to be degenerate (even
totally isotropic), hence its  dimension is even (not necessarily a multiple of 4) and
the signature of $g|_{\overline {T_x M}}$ is $(2k,2s,2k)$ where $2s= \dim \ker g $ (see \cite{MV}). We recall that a subspace of a para-quaternionic
vector space $(V,Q)$ is \textit{pure} if it contains no non zero $Q$-invariant subspace.
We  write then
\[ T_xM =\overline {T_xM} \oplus \mathcal D_x \]
 where $\mathcal D_x$ is any $\mathcal J^\epsilon$-invariant
pure supplement (the existence of such supplement is proved in \cite{MV}).

 Recall that if $M$
is a non degenerate submanifold of a pseudo-Riemannian manifold $(\widetilde M,
\widetilde g)$ and $T_x \widetilde{M} = T_xM \oplus T_x^\bot M$ is
the orthogonal decomposition of the tangent space $T_x \widetilde{M}$
at point $x\in M$ then the Levi-Civita covariant derivative
$\widetilde \nabla_X$ of the metric $\widetilde g$  in the
direction of a vector $X \in T_xM$ can be written as:
\[
\widetilde\nabla_X \equiv
\left( \begin{array}{cc}  \nabla_X & -A_X \\
A^t_X & \nabla^\bot_X \\
  \end{array} \right).
\]
that is
\begin{equation} \label{shape operator}
\widetilde\nabla_XY = \nabla_XY +h(X,Y), \qquad  \qquad
\widetilde\nabla_X\xi = -A^\xi X +\nabla^\bot_X\xi
\end{equation}
for any tangent (resp. normal) vector field $Y$ (resp. $\xi$)
on $M$. Here $\nabla_X$ is the covariant derivative of the induced
metric $g$ on $M$, $\nabla^\bot_X$ is the normal covariant
derivative in the normal bundle $T^\bot M$ which preserves the \it
normal metric \\\rm $g^\bot =\widetilde g|_{T^\bot M}$, $A^t_XY = h(X,Y)
\in T^\bot M$ where $h$ is the \textbf {second fundamental form}  and
$A_X\xi = A^\xi X$, where $A^\xi \in \text{\rm End } TM$ is the
\textbf {shape operator}  associated with a normal vector $\xi$.
\par

\begin{teor} \label{almost Hermitian is Hermitian if a certain 1-form vanishes}
Let $(M^{2m}, \mathcal{J}^{\epsilon},g)$, $m>1$, be an \aeH
 sub\-ma\-ni\-fold
of the para-quaternionic K\"ahler manifold  $(\widetilde
M^{4n},Q,\widetilde g)$. Then
\begin{enumerate}
\item  the almost $\epsilon$-complex structure  $\mathcal{J^\epsilon}$ is integrable if and only if the
 local 1-form
$\psi=\omega_3 \circ \mathcal{J^\epsilon} - \omega_2$ on $M^{2m}$
associated with an adapted basis $H=(J_\alpha)$ vanishes.
\item  $\mathcal{J^\epsilon}$ is integrable if one of the
following conditions holds:
\par
a) $\text{dim}(\mathcal D_x) > 2$ on an open dense set $U \subset
M$;
\par
b) $(M,J^\epsilon)$ is analytic and $\operatorname{dim}(\mathcal
D_x)
> 2$ at some point $x \in M$;
\end{enumerate}
\end{teor}

\begin{proof} (1)
Let proceed as in \cite{AM}, Theorem 1.1. Remark that if $(M,\mathcal{J^\epsilon})$ is an almost $\epsilon$-complex
submanifold of an almost $\epsilon$-complex manifold
$(\widetilde{M},J^\epsilon)$ then the restriction of the Nijenhuis
tensor $N_{J^\epsilon}$  to the submanifold $M$ coincides with the
Nijenhuis tensor $N_\mathcal{J^\epsilon}$ of the almost complex
structure $\mathcal{J^\epsilon}={J^\epsilon}_{|TM}$. Then
 for any $X,Y \in TM $, we can write
\[
\begin{array}{llll}
\frac{1}{2}N_\mathcal{J^\epsilon} (X,Y) &=&
[\mathcal{J^\epsilon} X,\mathcal{J^\epsilon} Y]
- \mathcal{J^\epsilon} [\mathcal{J^\epsilon} X,Y] - \mathcal{J^\epsilon} [X,\mathcal{J^\epsilon} Y] + \epsilon [X,Y]= \\

\frac{1}{2}N_{J^\epsilon}(X,Y) & = &
[\widetilde\nabla_{J^\epsilon X}(J^\epsilon Y)-
\widetilde\nabla_{J^\epsilon Y}(J^\epsilon X)] -J^\epsilon
[\widetilde\nabla_{J^\epsilon X}Y
-\widetilde\nabla_Y (J^\epsilon X)]\\

 & & -J^\epsilon [ \widetilde\nabla_X (J^\epsilon Y)
- \widetilde\nabla_{J^\epsilon Y}X ] + \epsilon [\nabla_X Y - \nabla_Y X]\\



 & = & (\widetilde\nabla_{J^\epsilon X}J^\epsilon)Y
-(\widetilde\nabla_{J^\epsilon Y}J^\epsilon)X + J^\epsilon
(\widetilde\nabla_Y J^\epsilon)X -J^\epsilon(\widetilde\nabla_X
J^\epsilon)Y
\end{array}
\]
and hence, from (\ref{nabla J})
\[
\begin{array}{ll}
\frac{1}{2}N_\mathcal{J^\epsilon}(X,Y)= &
-[\omega_3(\mathcal{J^\epsilon} X)- \omega_2(X)]J_2 Y + [-\epsilon
\omega_2(\mathcal{J^\epsilon} X)+
\omega_3(X)]J_3Y \\
& + [\omega_3(\mathcal{J^\epsilon} Y)- \omega_2(Y)]J_2 X  -
[-\epsilon \omega_2(\mathcal{J^\epsilon} Y)+\omega_3(Y)]J_3 X
\end{array}
\]
where $(J_1,J_2,J_3)$ is an  adapted local basis.
This
implies (1) in one direction.




Viceversa, let $N_\mathcal{J^\epsilon}(X,Y)=0, \; \forall X,Y \in T_xM$.
By applying $J_2$ to both members of  the above equality, this is equivalent to the identity
\begin{equation} \label{Nijenhuis equal zero}
\psi(X)Y + \epsilon \psi(\mathcal{J^\epsilon} X)\mathcal{J^\epsilon}Y= \psi(Y)X+ \epsilon \psi(\mathcal{J^\epsilon} Y)\mathcal{J^\epsilon}X, \quad \forall X,Y \in T_xM.
\end{equation}
Let assume that there exists a non zero vector $X \in T_xM$ such that $\psi(X) \neq 0$. We show that this leads to a contradiction.
Let consider a vector $0\neq Y \in T_xM$ which is not en eigenvector of $\mathcal{J^\epsilon}$ and such that  $span(X,\mathcal{J^\epsilon}X)  \cap  span(Y,\mathcal{J^\epsilon}Y)= {0}$.
It is easy to check that such a vector $Y$ always exists. Then the vectors in both sides  of  (\ref{Nijenhuis equal zero}) must be  zero which implies in particular that $\psi(X)=0$. Contradiction.

\par

(2) We assume that $\mathcal{J^\epsilon}$ is not integrable. Then
the 1-form $\psi = (\omega_3\circ \mathcal{J^\epsilon} - \omega_2)|_{TM}$
 is not identically zero, by (1). Denote by $a=g^{-1}\psi$ the
local vector field on $M$ associated with the 1-form $\psi$ and
let $a=\overline a + a'$ with  $\overline a \in \overline {TM}$ and
$a' \in \mathcal D$. Now we need the following
\begin{lemma} \label{dimension 0,2 of D where N non si annulla}
Let $(M^{2m},\mathcal{J^\epsilon},g), m>1$, be an \aeH
submanifold of a \pqK manifold  $(\widetilde M^{4n},Q,\widetilde g)$. Then in any point $x \in M^{2m}$ where the
Nijenhuis tensor $N(\mathcal{J^\epsilon})_x\neq 0$, or
equivalently the vector $a_x \neq 0$, any $\mathcal{J}^\epsilon$-invariant supplementary  subspace $\mathcal D_x$
is  spanned by $a'_x$ and $\mathcal{J^\epsilon}a_x'$:
$$\mathcal D_x = \text{\rm span}\{a_x',\mathcal{J^\epsilon}a_x'\}.$$
Moreover if $T_xM$ is not para-quaternionic  (i.e. $\dim \mathcal D_x \neq
0$) then $\psi(\overline{T_x M}) \equiv 0$.
\end{lemma}

\begin{proof}
Remark that
\begin{equation} \label{Nijenhuis}
\begin{array}{ll}
\frac{1}{2}N_\mathcal{J^\epsilon}(X,Y)&= -\psi(X)J_2Y + \epsilon
\psi(\mathcal{J^\epsilon}X)J_3 Y + \psi(Y) J_2X
- \epsilon \psi(\mathcal{J^\epsilon}Y) J_3X\\
& = -J_2 \{ \psi(X)Y + \epsilon
\psi(\mathcal{J^\epsilon}X)\mathcal{J^\epsilon} Y - \psi(Y) X -
\epsilon \psi(\mathcal{J^\epsilon}Y) \mathcal{J^\epsilon}X \},
\end{array}
\end{equation}
that is $N_\mathcal{J^\epsilon}(X,Y)  \in J_2TM \cap TM=
\overline {TM}$ for any  $X,Y \in TM$. Hence
\begin{equation} \label{Nijenhuis vector belongs to para-quaternionic component of tangent space}
\bigg[ \psi(X)Y + \epsilon
\psi(\mathcal{J^\epsilon}X)\mathcal{J^\epsilon} Y - \psi(Y) X -
\epsilon \psi(\mathcal{J^\epsilon}Y) \mathcal{J^\epsilon}X \bigg]
\in \overline{TM} \quad \forall X,Y \in TM.
\end{equation}
Taking  $X \in \overline {T_x M}$ and $ 0 \neq Y \in \mathcal D_x$ the first two terms of
(\ref{Nijenhuis vector belongs to para-quaternionic component of tangent space}) are in $\mathcal{D}_x$  and the last two in $\overline{T_xM}$. We  conclude
 that $\psi(\overline {T_x M}) \equiv 0$ if $\dim \mathcal D_x
\neq 0$. For $X=a=g^{-1}\psi$, since $g(a,J^\epsilon a)=0$, the last
condition says that
$$
b_Y:=|a|^2Y -\psi(Y)a - \epsilon
\psi(\mathcal{J^\epsilon}Y)\mathcal{J^\epsilon}a \in \overline{TM}
\qquad \qquad \forall \, ~Y\in TM
$$
Considering the $\mathcal D$-component of the vector $b_Y$  for $Y=\overline Y
\in \overline{TM}$ and $Y=Y' \in \mathcal D$ respectively, we get
the equations:
\begin{equation} \label{generatori di D prima}
- \psi(\overline Y)a' - \epsilon \psi(\mathcal{J^\epsilon}
\overline Y)\mathcal{J^\epsilon} a'=0 \, , \qquad \qquad \forall
\,~ \overline Y\in \overline{TM}
\end{equation}
\begin{equation}\label{generatori di D seconda}
|a|^2Y' -\psi(Y')a' - \epsilon \psi(\mathcal{J^\epsilon}
Y')\mathcal{J^\epsilon} a'=0 \qquad \qquad \forall \, ~Y' \in
\mathcal D.
\end{equation}
The last equation shows that $\mathcal D_x =
\{a',\mathcal{J^\epsilon} a'\}$ when $a\neq 0$ (whereas
(\ref{generatori di D prima}) confirms that $\psi(\overline {T_x
M}) \equiv 0$ when $\dim \mathcal D \neq 0$). Observe that $a'$ is never an eigenvector of the para-complex structure $\mathcal{J}$.

\end{proof}
\textit{Continuing the proof of Theorem (\ref{almost Hermitian is Hermitian if a certain 1-form vanishes})}:
The Lemma implies statements (2a)  and (2b) since in the analytic
case the set $U$ of points where the analytic vector field $a\neq
0$ is open (complementary of the close set where  $a=0$) and dense
(since otherwise it would exist  an open  set $\widetilde U$ with $a(\widetilde
U)=0$ which, by the analiticy of $a$ it would imply $a=0$
everywhere) and $\operatorname{dim}\mathcal D_x \le 2$ on $U$.
\end{proof}

From (\ref{Nijenhuis}) it follows the
\begin{coro} \label{in pure (para)complex tangent space Nijenhuis tensor is zero} In case $T_xM$ is pure $\epsilon$-complex i.e. $\overline{T_xM}=\bf{0}$
 in an open dense set in $M$ than the
almost Hermitian submanifold is Hermitian.
\end{coro}
This is a generalization of  the 2-dimensional case where clearly, by the non degeneracy hypotheses, $T_xM$ is
pure for any $x \in M$.

\begin{defi}
A submanifold $M$ of an almost para-quaternionic manifold $(\widetilde M,Q)$  is an \textbf{almost para-quaternionic submanifold} if its tangent bundle
is $Q$-invariant. Then $(M, Q|_{TM})$ is an almost  para-quaternionic manifold.
\end{defi}

The following proposition is the extension to the para-quaternionic case of a basic result in
quaternionic case.
\begin{prop} \label{para-quaternionic submanifold in para-quaternionic
Kaehler manifold is Kaehler and totally geodesic} A non degenerate almost
para-quaternionic submanifold $M^{4m}$ of a para-quaternionic K\"ahler
ma\-ni\-fold $(\widetilde{M}^{4n},Q,\widetilde g)$ is a totally geodesic  para-quaternionic K\"ahler
submanifold.
\end{prop}
\begin{proof}
Let $A$ be the shape operator of the para-quaternionic submanifold.
Then, for any $X,Y \in \Gamma(TM), \; \xi \in \Gamma(T^\perp M)$,
\[
\begin{array}{ll}
& \widetilde{g}(A^\xi (J_\alpha X),Y) =-\widetilde{g}(\widetilde\nabla_{J_\alpha X} \xi
,Y)=-\widetilde{g}(\widetilde\nabla_Y \xi,J_\alpha X)\\
&=\widetilde{g}(\xi,\widetilde\nabla_Y(J_\alpha X))= \widetilde{g}(\xi,(\widetilde\nabla_Y \widetilde
J_\alpha) X + \widetilde J_\alpha \widetilde\nabla_Y X).
\end{array}
\]
Moreover
\[
\begin{array}{ll}
\widetilde{g}(\xi, \widetilde J_\alpha \widetilde\nabla_Y X) & = -\widetilde{g}(\widetilde J_\alpha
\xi,
\widetilde\nabla_YX)=-\widetilde{g}(\widetilde J_\alpha \xi, \widetilde\nabla_XY -[X,Y])\\
&=-\widetilde{g}(\widetilde J_\alpha \xi, \widetilde\nabla_X Y)=\widetilde{g}(\xi, \widetilde J_\alpha
\widetilde\nabla_X
Y)=\widetilde{g}(\xi,\widetilde\nabla_X(\widetilde J_\alpha Y)-(\widetilde\nabla_X \widetilde J_\alpha) Y)\\
&= \widetilde{g}(\xi,\widetilde\nabla_X(J_\alpha Y))= -\widetilde{g}(\widetilde\nabla_X
\xi,J_\alpha Y)=  -\widetilde{g}(J_\alpha A^\xi X,Y)
\end{array}
\]
and
\[
\begin{array}{ll}
 \widetilde{g}(\xi,(\widetilde\nabla_Y \widetilde J_\alpha) X)=\widetilde{g}(\xi,- \epsilon_\beta \omega_\gamma (Y)J_\beta X + \epsilon_\gamma \omega_\beta(Y)
 J_\gamma)=0
\end{array}
\]
since $J_\beta X,J_\gamma X \in \Gamma(TM)$. It follows that
$AJ_\alpha= -J_\alpha A, \quad \alpha=1,2,3$. Computing
$AJ_\alpha=-J_\alpha A=- \epsilon_3 \epsilon_\alpha J_\beta
J_\gamma A= -\epsilon_3 \epsilon_\alpha A J_\beta J_\gamma
=-(\epsilon_3 \epsilon_\alpha)^2 A J_\alpha= -AJ_\alpha$ we get $A= 0$ i.e. $h=0$.
Now it is immediate to deduce that $(M^{4m}, Q|_{TM},g)$ is also para-quaternionic K\"ahler.
\end{proof}

\begin{coro} Let $(M^{4k},\mathcal{J^\epsilon},g)$ be an
\aeH submanifold of di\-men\-sion $4k$ of a  \pqK manifold $\widetilde
M^{4n}$. Assume that the set $U$ of points $x\in M$ where the
Nijenhuis tensor of $\mathcal{J^\epsilon}$ is not zero is open and
dense in $M$ and that, $\forall x \in U, \; \overline{T_xM}$ is non
degenerate. Then $M$ is a totally geodesic \pqK submanifold.
\end{coro}

\begin{proof} As in  \cite{AM} by taking into account that, by the non
degeneracy hypotheses of $\overline{T_xM}$, it is necessarily $\operatorname{dim}\mathcal D_x=0$.
\end{proof}


\section{Almost $\epsilon$-K\"ahler, $\epsilon$-K\"ahler and totally $\epsilon$-complex
submanifolds}
\begin{defi} The \aeH
submanifold $(M^{2m},\mathcal{J^\epsilon},g)$ of a \pqK manifold
$(\widetilde M^{4n},Q,\widetilde g)$ is called \textbf{
almost $\epsilon$-K\"ahler} (resp., \textbf{ $\epsilon$-K\"ahler})
if the K\"ahler form $F= F_1 |_{TM}=g\circ \mathcal{J^\epsilon}$
is closed (resp. parallel).
Moreover $M$ is called \textbf{totally $\epsilon$-complex} if
$$
J_2T_xM \perp T_xM \qquad \qquad \forall \, x \in M
$$
where $(J_1,J_2,J_3)$ is an adapted basis (note that $J_2T_xM
\perp T_xM \Leftrightarrow J_3T_xM \perp T_xM$).
\end{defi}
For a study of (almost)-K\"ahler and totally complex submanifolds of a qua\-ter\-nio\-nic manifold see
 \cite{AM},\cite{7},\cite{11},\cite{19}.

 In case $\widetilde M$ is the $n$-dimensional para-quaternionic numerical space $\widetilde{\HH}^n$, the pro\-to\-type of flat para-quaternionic K\"ahler spaces (see \cite{ MV1}),
typical examples of such submanifolds are the flat K\"ahler (resp. para-K\"ahler) submanifolds $M^{2k}= \C^k$ (resp. $\widetilde{\C}^k$) obtained by choosing the first $k$ para-quaternionic coordinates  as com\-plex (resp. para-complex) numbers and the remaining $n-k$ equals to zero. In case $\widetilde M^{4n}=\widetilde{\HH}P^n$ is the para-quaternionic projective space endowed with the standard para-quaternionic K\"ahler metric (see \cite{DJS}), examples of  non flat  K\"ahler (resp. para-K\"ahler) submanifolds are given by the immersions of the projective complex (resp. para-complex) spaces $\C P^{k-1}$ (resp. $\widetilde \C P^{k-1}$)  induced by the  immersions considered above in the flat case.

From (\ref{nabla J}) one has
\begin{equation}\label{contition to be epsilon-Kaehel}
\begin{array}{l}
(\nabla_X \mathcal{J^\epsilon})Y = \big[-\omega_3(X)Id
- \epsilon \omega_2(X) \mathcal{J^\epsilon}\big]\big[J_2Y\big]^T \qquad  \qquad X,Y \in TM.
\end{array}
\end{equation}
and then, by arguing as in \cite{AM}, the following theorem is deduced.
\begin{teor}\label{almost Hermitian to be Kaehler}
Let $(\widetilde M^{4n},Q,\widetilde
g)$ be a \pqK manifold.\\
1)  A totally $\epsilon$-complex submanifolds  of $\widetilde M$ is $\epsilon$-K\"ahler.\\
2) If $\nu\neq 0$, for an \aeH submanifold $(M^{2m},\mathcal{J^\epsilon},g), \, m >1,$ of
$\widetilde M$  the following conditions are equivalent:\\
\indent $ k_1)$ $M$ is $\epsilon$-K\"ahler,\\
\indent $ k_2)$ ${\omega_2}|_{T_xM}={\omega_3}|_{T_xM}=0
\qquad \qquad \forall \, x \in M$,\\
\indent $ k_3)$ $M$ is totally $\epsilon$-complex.
\end{teor}
\begin{proof}
The first statement follows from  (\ref{contition to be epsilon-Kaehel}).
The second statement is proved in \cite{AC1} Proposition 20.
\end{proof}

\begin{teor} \label{second fundamental form of a Kaehler submanifold} Let $(\widetilde M^{4n},Q,\widetilde
g)$ be a \pqK manifold with non vanishing reduced scalar curvature
$\nu$ and $(M^{2m},\mathcal{J^\epsilon},g)$ an \aeH sub\-ma\-ni\-fold of
$\widetilde M^{4n}$.
\begin{enumerate}
\item[a)] If $(M^{2m},\mathcal{J^\epsilon},g)$ is $\epsilon$-K\"ahler then the second fundamental
form $h$ of $M$ sa\-tis\-fies the identity
\begin{equation} \label{2 fundamental form in Kaehler submanifold}
h(X,\mathcal{J^\epsilon} Y)=h(\mathcal{J^\epsilon} X,Y)=J^\epsilon
h(X,Y) \qquad \qquad \forall \, X,Y \in TM.
\end{equation}
In particular $h(\mathcal{J^\epsilon} X,\mathcal{J^\epsilon} Y)= \epsilon
h(X,Y)$.
\item[b)] Conversely, if the identity (\ref{2 fundamental form in
Kaehler submanifold}) holds on an \aeH submanifold $M^{2m}$ of
$\widetilde M^{4n}$ then it is either a $\epsilon$-K\"ahler
submanifold or a para-quaternionic (K\"ahler) submanifold and
these cases cannot happen simultaneously.
\end{enumerate}
\end{teor}
\begin{proof} (a)
 Let $(M^{2m},\mathcal{J^\epsilon},g)$ be an \aeH submanifold of
$\widetilde M $.  By  (\ref{nabla J}),
\begin{equation} \label{nabla J1}
\begin{array} {ll}
 (\widetilde\nabla_X J^\epsilon)Y &= (\nabla_X \mathcal{J^\epsilon})Y + h(X,\mathcal{J^\epsilon}Y) - J^\epsilon h(X,Y)\\
 &=-\omega_3(X)J_2Y - \epsilon \omega_2(X)J_3Y,  \qquad  \qquad X,Y \in TM.
\end{array}
\end{equation}
From  Theorem (\ref{almost Hermitian to be Kaehler}), we get
\[
0=(\nabla_X \mathcal{J^\epsilon})Y + h(X,\mathcal{J^\epsilon}Y) -
J^\epsilon h(X,Y), \qquad \forall X,Y \in TM
\]
and, from $(\nabla_X \mathcal{J^\epsilon})Y=0$   it is clear that
if $(M,\mathcal{J^\epsilon})$ is $\epsilon$-K\"ahler then (\ref{2
fundamental form in Kaehler submanifold}) holds.
\par
(b) Conversely, let assume that (\ref{2 fundamental form in
Kaehler submanifold}) holds on the \aeH submanifold
$(M,\mathcal{J^\epsilon},g)$. Then for any $X,Y \in T_xM$, from
(\ref{nabla J1}) we have
\[(\nabla_X \mathcal{J^\epsilon})Y=(\widetilde \nabla_X J^\epsilon)Y.\]
Hence, $\forall \, X,Y\in T_xM$,
\[
\begin{array}{l}
(\nabla_X \mathcal{J^\epsilon})Y = -\omega_3(X)J_2Y - \epsilon \omega_2(X)J_3Y =(-\omega_3(X)Id - \epsilon \omega_2(X) \mathcal{J^\epsilon})J_2Y \in
T_xM. 
\end{array}
\]
Then, either $J_2 T_xM=T_xM$ i.e. $T_xM$ is a para-quaternionic vector space or $\omega_2|_x=\omega_3|_x=0$
and by Theorem  (\ref{almost Hermitian to be Kaehler}) the two conditions cannot happen simultaneously.
The set $M_1=\{x \in M  \; | \; J_2 T_xM=T_xM \}$ is a closed subset  and the complementary open subset $M_2=\{ x \in M \; | \; \omega_2|_x=\omega_3|_x=0 \}$
is a closed subset as well since, from
 Theorem  (\ref{almost Hermitian to be Kaehler}), $M_2=\{ x \in M \; | \; J_2 T_xM \perp T_xM \}$. Then, either $M_2=0$ and $M=M_1$ is a para-quaternionic  K\"ahler submanifold
or  $M_1=0$ and $M=M_2$ is
  $\epsilon$-K\"ahler.
\end{proof}

\begin{coro} \label{totally geodesic Hermitian submanifold is
either Kaehler or a pQK submanifold}  A totally geodesic \aeH
submanifold $(M,\mathcal{J^\epsilon},g)$ of a para-quaternionic
K\"ahler manifold $(\widetilde M^{4n},Q,\widetilde g)$ with
$\nu \neq 0$ is either a $\epsilon$-K\"ahler submanifold or a
para-quaternionic submanifold and these conditions cannot happen
simultaneously.
\end{coro}
\begin{proof} The  statement follows directly from Theorem
(\ref{second fundamental form of a Kaehler submanifold}) since
(\ref{2 fundamental form in Kaehler submanifold}) certainly holds
for a totally geodesic submanifold ($h=0$).
\end{proof}


The following results have been  proved in \cite{AC1}.

\begin{prop} \label{shape operator anti commutes with (para) complex structure}
{\rm{(\cite{AC1})}} The shape operator $A$ of an $\epsilon$-K\"ahler  submanifold
    $(M^{2m}, \mathcal{J^\epsilon}, g)$ of a para-quaternionic
K\"ahler manifold $(\widetilde
 M^{4n},Q,\widetilde g)$ anticommutes with $\mathcal{J^\epsilon}$, that
is $A \mathcal{J^\epsilon}=-\mathcal{J^\epsilon} A$.
\end{prop}

\begin{coro} \label{Kaehler and para-Kaehler submanifolds are minimal}
{\rm{(\cite{AC1})}}  Any $\epsilon$-K\"ahler  submanifold of a para-quaternionic
K\"ahler ma\-ni\-fold is minimal.
\end{coro}


We conclude this section with the following result concerning almost $\epsilon$-K\"ahler submanifolds.

\begin{teor} \label{almost Kaehler is Kaehler} Let $(\widetilde M^{4n},Q,\widetilde
g)$ be a para-quaternionic K\"ahler manifold with non vanishing
reduced scalar curvature $\nu$. Then any   almost
$\epsilon$-K\"ahler submanifold \\ $(M^{2m},\mathcal{J^\epsilon},g)\,$ of $\widetilde M$ is
$\epsilon$-K\"ahler.
\par
\end{teor}

\begin{proof}
 By identity (\ref{exterior differential
of fundamental 2-form}), the condition that the  K\"ahler form
$F={F_1}_{|_M}$ is closed can be written as
\begin{equation} \label{condition for a Kaehler form to be closed}
 F^T_2\wedge\omega_3^T= \epsilon F^T_3\wedge\omega^T_2 \, ,
\end{equation}
where $F^T_\alpha, \omega^T_\alpha$ are the restriction of the
forms $F_\alpha,\omega_\alpha$ to $M$. We will prove that (\ref{condition for a Kaehler form to be closed})  implies
integrability.

Let suppose  that there exists a point $x$ of the almost $\epsilon$-K\"ahler submanifold $M$
where ${N_\mathcal{J^\epsilon}}_{|_x} \neq 0$.
From  Lemma (\ref{dimension 0,2 of D where
N non si annulla}) then $\dim \mathcal D_x=0$ or 2 and from  (\ref{Nijenhuis vector belongs to para-quaternionic component of tangent space}) one has  $\dim \overline{T_xM}>0$.

Let first consider  the case that $T_xM= \overline{T_xM}$. Observe that by hypotheses $T_xM$ is non degenerate (than $\dim T_xM=4k$).
By applying  (\ref{condition for a Kaehler form to be closed}) to the triple $(X,J_2X,J_1X)$ for $X \in T_xM$ no eigenvector of any compatible para-complex structure in $Q$,
we have
\[
\begin{array}{ll}
&(F^T_2\wedge\omega^T_3)(X,J_2X,J_1X)=-\Vert X\Vert^2\omega^T_3(J_1X)=-\omega^T_3(J_1X)\\
= &\epsilon(F^T_3\wedge\omega^T_2)(X,J_2X,J_1X)=\epsilon
F^T_3(J_2X,J_1X)\omega^T_2(X)= -\Vert X\Vert^2\omega^T_2(X).
\end{array}
\]
Hence $\omega^T_3= \epsilon \omega^T_2 \circ \mathcal{J^\epsilon}$ and, from Theorem (\ref{almost Hermitian is Hermitian if a certain 1-form vanishes}), it follows that
${N_\mathcal{J^\epsilon}}_{|_x}= 0$. Contradiction.

\vskip 0.3cm

Let now suppose that  $\textrm{codim} \;\overline{T_xM}=2$.
From Lemma (\ref{dimension 0,2 of D where
N non si annulla}) it is $\psi(\overline{T_xM})=0$.
If $\overline{T_xM}$ is non degenerate, calculating  both sides of equation (\ref{condition
for a Kaehler form to be closed}) on vectors $X,J_2X,Y$, where $X$
is a unit vector from $\overline{T_xM}$ and $Y \in \mathcal D_x$ is the
$\mathcal{J^\epsilon}$-invariant  orthogonal complement to $\overline{T_xM}$ in $T_xM$
 we get
$$
(F^T_2\wedge\omega^T_3)(X,J_2X,Y)=\epsilon \omega^T_3(Y)=-(F^T_3
\wedge \eta) (X,J_2X,Y)=0.
$$
Hence, $\omega^T_3(\mathcal D_x)=0=\omega_2^T(\mathcal D_x)$ which implies that ${N_\mathcal{J^\epsilon}}_{|_x}= 0$. Contradiction.
In case that  $\overline{T_xM}$  is degenerate (even totally isotropic) and $\dim \mathcal D_x =2$ with $\mathcal D_x$ any $\mathcal{J^\epsilon}$-invariant
 complement to $\overline{T_xM}$ in $T_xM$, by  evaluating (\ref{condition for a Kaehler form to be closed}) on the triple $(Y, \mathcal{J}^\epsilon Y, X)$ with $\{Y,\mathcal{J}^\epsilon
Y\} $ any  basis of $\mathcal D_x$   and $X \in \ker g_{\overline{T_xM}}$
it is
\[
F^T_2 \wedge \omega^T_3(X,Y,\mathcal{J}^\epsilon Y)= <J_2X,Y>
\omega^T_3(\mathcal{J}^\epsilon Y) - <J_3Y,X> \omega^T_3(Y);
\]
\[
 \epsilon(F^T_3 \wedge \omega^T_2)(X,Y,\mathcal{J}^\epsilon Y)= \epsilon[<J_3X,Y> \omega^T_2(\mathcal{J}^\epsilon Y) -\epsilon  <J_2Y,X> \omega^T_2(Y)]
\]
i.e.
\[
<J_2X,Y> [\omega_3(\mathcal{J}^\epsilon Y) -   \omega_2(Y)]   - <J_3Y,X>
[\omega_3(Y) -\epsilon \omega_2(\mathcal{J}^\epsilon Y)]=0.
\]

Then, considering the non degeneracy of $T_xM$,  the only solution is given by
\[
[\omega^T_3 \circ \mathcal{J}^\epsilon Y -   \omega^T_2]= [\omega^T_3 - \epsilon \omega^T_2
\circ \mathcal{J}^\epsilon Y)]=0, \quad \forall Y \in \mathcal D_x
\]
i.e. $\psi(\mathcal D_x) = 0$  which leads again to the contradiction  that ${N_\mathcal{J^\epsilon}}_{|_x}= 0$.
\end{proof}

We state the following corresponding  result regarding quaternionic geometry:
\begin{teor} \label{almost Kaehler in quaternionic case  are Kaehler except in 4 dimension}  Let $(\widetilde M^{4n},Q,\widetilde
g)$ be a quaternionic K\"ahler manifold with non zero
reduced scalar curvature $\nu$. Then any   almost
K\"ahler submanifold $(M^{2m},\mathcal{J},g)$, $\;n \neq 2$ of $\widetilde M$ is
K\"ahler.
\par
\end{teor}
\begin{proof}
Here the condition for a submanifold to be  almost-K\"ahler is given
by the equation
\begin{equation} \label{almost_kaehler in quaternionic geometry}
F^T_2\wedge\omega^T_3 =
F^T_3\wedge\omega^T_2.
\end{equation}
The result  for dimension greater that 6 has been given in \cite{AM}.

By applying  the proof  of our Theorem (\ref{almost Kaehler is Kaehler}) to the other cases and considering that
in quaternionic case the  metric in each subspace of $T_xM$ is positive definite,  the conclusion follows.
With respect to the para-quaternionic case the difference concerning the dimension 4 follows from the fact that,
 in a point  $x \in M$ where the tangent space $T_xM$ is a 4 dimensional (Euclidean) quaternionic
vector space, the equation (\ref{almost_kaehler in quaternionic geometry}) admits the non trivial solution
$(\omega^T_2, \omega^T_3=\omega^T_2 \circ \mathcal{J})$ which  does not imply ${N_\mathcal{J^\epsilon}}_{|_x}= 0$
that  happens iff $\omega^T_3 \circ \mathcal{J} -
\omega^T_2  = 0$.
\end{proof}

\section{Maximal $\epsilon$-K\"ahler submanifolds of a para-quaternionic
K\"ahler manifold}

\subsection {The shape tensor $C$ of a $\epsilon$-K\"ahler submanifold}

Let $(M^{2n}, \mathcal{J^\epsilon},g)$ be a $\epsilon$-K\"ahler
submanifold of maximal possible dimension $2n$ of a
para-quaternionic K\"ahler manifold $(\widetilde M^{4n}, Q,\widetilde g)$ with $\nu \neq0$. We fix an  adapted basis $(J_1= J^\epsilon,
J_2,J_3= J_1 J_2, \; J_1^2 = \epsilon Id, J_2^2 = Id, \;
\mathcal{J^\epsilon}= J_1 |_{TM})$ of $Q$ and assume that it is defined on a neighbourhood of $M^{2n}$ in
$\widetilde M^{4n}$. From Theorem (\ref{almost Hermitian to be Kaehler}), the submanifold $M$ is totally $\epsilon$-complex.
We have the orthogonal decomposition
\begin{equation} \label{decomposition of tangent space}
T_x \widetilde{M} = T_xM \oplus J_2T_xM \qquad \qquad \forall \,
x\in M.
\end{equation}

Since $ {\omega_2}|_{T_xM}={\omega_3}|_{T_xM}=0  \quad  \forall
\, x \in M,$ then the following equations hold:
\begin{equation} \label{nabla J1,J2,J3}
 \widetilde\nabla_X{J_1}=0  \quad , \quad
\widetilde\nabla_X{J_2}=  \epsilon \omega(X)J_3 \quad, \quad
\widetilde\nabla_X{J_3}= \omega(X)J_2 \qquad \forall \, X \in TM
\end{equation}
where $\omega= \omega_1|_{TM}$ is a 1-form. We identify the normal bundle $T^\bot M$ with the tangent bundle
$TM$ using $J_2$ (note that $J_2^{-1}= J_2$):
$$
\begin{array}{cccc}
\varphi = {J_2}|_{T^\bot M}: & T_x^\bot M & \rightarrow &  T_xM\\
& \xi & \mapsto & J_2\xi\, .
\end{array}
$$
Then the second fundamental form $h$ of $M$ is identified with the
tensor field
$$
C=J_2\circ h  \in TM\otimes S^2T^*M
$$
and the normal connection $\nabla^\bot$ on $T^\bot M$ is
identified with a linear connection
\mbox{$\nabla^N=J_2\circ\nabla^\bot\circ J_2$} on $TM$.
We will call $C$ the \textbf{shape tensor} of the
$\epsilon$-K\"ahler sub\-ma\-ni\-fold $M$.
Note that $C$ depends on the adapted basis $(J_\alpha)$ and it is
defined only locally. We recall (see \cite{MV1}) that the 3-dimensional vector space
$Q_x \subset End(T_x\widetilde M)$ has a natural pseudo-Euclidean norm defined by $L^2= -||L||^2 Id, \; L \in Q$.
W.r.t. the adapted basis above,  if  $L= aJ_1 + bJ_2+ cJ_3 \in Q$ then  $||L||^2=a^2 -b^2 -c^2$ if $\epsilon =-1$  and $||L||^2= -a^2
-b^2 +c^2$ if $\epsilon =1$. Then  if $(J_\alpha')$ is another adapted basis obtained by the
pseudo-orthogonal  transformation, represented in the base
$(J_1,J_2, J_3)$, by the following matrices $B_\epsilon \in
SO(2,1)$
\begin{equation}\label{matrix of change of adapted basis on Kaehler submanifold }
 B_{-1}=\left(
\begin{array}{ccc}
1 & 0 & 0 \\ 0 & cos \theta & -sin \theta \\ 0 & sin \theta & cos
\theta
\end{array}
\right), \qquad  B_1=\left(
\begin{array}{ccc}
1 & 0 & 0 \\ 0 & cosh \theta & sinh \theta \\ 0 & sinh \theta &
cosh \theta
\end{array}
\right)
\end{equation}
then the shape
tensor transforms as
$$
C \mapsto C' = J_2' \circ h=    \cos\theta C+ \sin\theta
\mathcal{J^\epsilon} \circ C \, \quad (\text{resp. } \;
C'=\cosh\theta C+
 \sinh\theta \mathcal{J^\epsilon} \circ C) \,  .
$$
In the following $(E_i), \; i= 1,\ldots, 2n$ will be an orthonormal basis of $T_xM$ and we will use the notation $\mu_i=<E_i,E_i>$.
\begin{lemma} \label{max Kaehler proprieties}
One has
\begin{itemize} \item[(1)] For any $X \in TM$
the endomorphism $C_X$ of $TM$ is symmetric  and
\par
 $C_X=-A^{\varphi^{-1}X} = - A^{{J_2}X}$
where $A^\xi$ is the shape operator, defined in {\rm (\ref{shape
operator})}
\par
(Note that $ C_{J_2\xi}=-A^\xi \quad , \forall \, \xi \in T^\bot
M) $.
\item[(2)] $\nabla^N_X = \nabla_X - \epsilon \omega(X)\mathcal{J^\epsilon}, \quad  \quad X \in TM$.
\item[(3)] The curvature of the connection $\nabla^N$ is given by
\[\qquad \qquad R^N_{XY}=R_{XY} - \epsilon d\omega(X,Y)\mathcal{J^\epsilon}.\]
\item[(4)]  $\{ C_X , \mathcal{J^\epsilon}\} = C_X \circ \mathcal{J^\epsilon} + \mathcal{J^\epsilon} \circ C_X = 0$ and hence
$tr \, C = \sum_{2n} \mu_i C_{E_i}E_i = 0$.
\item[(5)] The tensors $gC$ and $gC\circ \mathcal{J^\epsilon}$ defined by
\[gC(X,Y,Z) = g(C_XY,Z), \qquad   \qquad (gC\circ \mathcal{J^\epsilon})(X,Y,Z) = gC(\mathcal{J^\epsilon}X,Y,Z)\]
are symmetric, i.e. both $gC$ and $gC\circ \mathcal{J^\epsilon} \in
S^3T^*M$.
\end{itemize}
\end{lemma}

\begin{proof} (1) Using (\ref{decomposition of tangent space}) and (\ref{nabla J1,J2,J3}), for any $X,Y,Z \in TM$
one has
$$
\begin{array}{ll} \langle C_XZ,Y\rangle &=  \langle J_2 \circ h(X,Z),Y\rangle= -\langle h(X,Z),J_2Y\rangle
= -\langle \widetilde \nabla_X(Z),J_2 Y\rangle  \\
& = \langle \widetilde \nabla_X(J_2 Y),Z\rangle=  \langle
 (\widetilde\nabla_X J_2) Y + J_2 \widetilde \nabla_XY,Z\rangle\\
& = \langle  \epsilon \omega(X) J_3 Y + J_2 \widetilde
\nabla_XY,Z\rangle=\langle J_2\widetilde\nabla_XY,Z\rangle\\
& = -\langle \widetilde\nabla_XY,J_2Z\rangle
 = -\langle h(X,Y),J_2 Z\rangle = \langle C_XY,Z\rangle \, .
\end{array}
$$
Moreover, for any $X,Y,Z \in TM$,
\[
\langle -A^{J_2X}Y,Z \rangle =-\langle h(Y,Z),J_2X\rangle =
\langle J_2h(Y,Z),X\rangle
= \langle C_YZ,X \rangle=\langle Z, C_XY\rangle.
\]
This implies that $C_{X}=-A^{J_2X}$.

(2) Denoting by $[\;]^\perp$ the projection on $T^\perp M$ of a
vector in $\widetilde {TM}$, we have
$$
\begin{array}{ll}
\nabla_X^NY & = J_2 \nabla_X ^\perp (J_2Y)= J_2 [\widetilde
\nabla_X (J_2Y)]^\perp= J_2 [(\widetilde \nabla_X J_2)Y+ J_2
\widetilde \nabla_X Y]^\perp\\
&= J_2 [( \epsilon \omega(X) J_3Y + J_2(  \nabla_X Y+
h(X,Y))]^\perp= - \epsilon \omega(X)JY + \nabla_XY.
\end{array}
$$

$$ \begin{array}{lll}
(3)\; R^N_{XY}Z &=&[\nabla_X - \epsilon
\omega(X)\mathcal{J^\epsilon},
\nabla_Y- \epsilon \omega(Y)\mathcal{J^\epsilon}](Z)-\nabla_{[X,Y]}Z + \epsilon \omega([X,Y])\mathcal{J^\epsilon}Z \\
&=&R_{XY}Z +\nabla_X[ -\epsilon \omega(Y) \mathcal{J^\epsilon}Z]

-\epsilon \omega(X)\mathcal{J^\epsilon} \nabla_YZ + \omega(X)
\omega(Y) {\mathcal{J^\epsilon}}^2Z +  \\  & &  \epsilon
\nabla_Y[\omega(X) \mathcal{J^\epsilon}Z] + \epsilon
\omega(Y)\mathcal{J^\epsilon}
\nabla_X Z - \omega(X) \omega(Y) {\mathcal{J^\epsilon}}^2Z + \epsilon \omega([X,Y])\mathcal{J^\epsilon}Z\\


 &=&R_{XY} Z - \epsilon \{X\cdot\omega(Y)-Y\cdot\omega(X)-\omega([X,Y])\}\mathcal{J^\epsilon}Z\\

 &=& R_{XY}Z -\epsilon d\omega(X,Y)\mathcal{J^\epsilon}Z. \end{array}
$$

(4) By using (\ref{2 fundamental form in Kaehler submanifold}) we
get
$$
C_X (\mathcal{J^\epsilon} Y) = J_2h(X,\mathcal{J^\epsilon}Y) =
J_2J^\epsilon h(X,Y) = -\mathcal{J^\epsilon} C_XY\, .
$$
Since  $C_X = -\mathcal{J^\epsilon} \circ C_X \circ
\mathcal{J^\epsilon}^{-1}$, then $tr \, C_X=0 \; \; \forall X \in
TM$, which implies $tr \, C=0$.

(5) The first statement follows from (1) and the symmetry of
$h$. Using (4) we prove the second one:
\[
\begin{array}{ll}
(gC\circ \mathcal{J^\epsilon})(X,Y,Z)&=gC(\mathcal{J^\epsilon}X,Y,Z) =
\langle C_{\mathcal{J^\epsilon} X}Y,Z\rangle = \langle C_Y(\mathcal{J^\epsilon} X),Z\rangle \\

& = - \langle \mathcal{J^\epsilon} C_YX,Z \rangle = \langle
C_YX,\mathcal{J^\epsilon} Z \rangle = \langle C_Y
(\mathcal{J^\epsilon}Z),X\rangle\\

&= \langle C_{\mathcal{J^\epsilon} Z}Y,X\rangle=(gC\circ \mathcal{J^\epsilon})(Z,Y,X).
\end{array}
\]
Moreover from (1) it is $(gC\circ \mathcal{J^\epsilon})(X,Y,Z)=
(gC\circ \mathcal{J^\epsilon})(X,Z,Y)$.
\end{proof}

We denote by $\nabla'$ the linear connection in a tensor bundle
which is a tensor product of a tangent tensor bundle of $M$ and a
normal tensor bundle whose con\-nec\-tions are respectively $\nabla$
and $\nabla^\bot$. For example, if $k$ is a section of the bundle
$T^\bot M\otimes S^2T^*M$ then $(\nabla'_Xk)(Y,Z)= \nabla_X^\bot(k(Y,Z))-k(\nabla_XY,Z) -
k(Y,\nabla_XZ)$.
By using (2) of Lemma (\ref{max Kaehler proprieties}), we get
\[
\begin{array}{ll}
J_2(\nabla'_Xh)(Y,Z) &=J_2\{\nabla_X^\perp[h(Y,Z)]-h(\nabla_XY,Z)-h(Y,\nabla_XZ\}=\\

&=\nabla_X^N[J_2 h(Y,Z)]-C_{\nabla_XY}Z-C_Y \nabla_XZ\\

&= (\nabla_X^NC)_Y Z + C_{\nabla_X^N Y}Z +C_Y \nabla_X^N Z
-C_{\nabla_X Y}Z-C_Y \nabla_X Z\\

\end{array}
\]
hence for the covariant
derivative of the second fundamental form we have:
\begin{equation}\label{derivative of second fundamental form in Kaehler submanifolds}
J_2(\nabla'_Xh)(Y,Z) =(\nabla_X^N C)_Y Z + 2 \epsilon \omega(X)
\mathcal{J^\epsilon}C_Y Z =(\nabla_XC)_YZ + \epsilon \omega(X) \mathcal{J^\epsilon} C_Y Z.
\end{equation}

Denote by
$S_\mathcal{J^\epsilon} = \{A \in \text{\rm End} TM,
\{A,\mathcal{J^\epsilon}\} = 0, \; g(AX,Y) = g(X,AY) \}
$ the bundle of symmetric endomorphisms of $TM$, which anticommutes
with $J$ and by $S^{(1)}_\mathcal{J^\epsilon} = \{ A \in \text{\rm Hom}(TM,
S_\mathcal{J^\epsilon})=T^*M\otimes S_\mathcal{J^\epsilon}, \;
A_XY = A_YX \}$ its first prolongation. Then conditions (4), (5) can be reformulated as follows.
\begin{coro} \label{derivative of C}
The tensor $C = J_2h$ belongs to the space
$S^{(1)}_\mathcal{J^\epsilon}$ and its covariant derivative is
given by
$$
\nabla_XC=J_2\nabla'_Xh - \epsilon \omega(X) \mathcal{J^\epsilon} \circ C \, .
$$
\end{coro}

\subsection{Gauss-Codazzi-Ricci equations}

Let $M$ be a submanifold of a pseudo-Riemannian ma\-ni\-fold $\widetilde M$
and $
\widetilde{R}_{XY}=R^{TT}_{XY} + R^{\perp T}_{XY} +
R^{T\perp}_{XY} + R^{\perp\perp}_{XY }$ the decomposition of the curvature operator $\widetilde R_{XY},
X,Y \in T_xM$ of the manifold $\widetilde M$ according to the
decomposition
$$ \text{\rm End} (T_x\widetilde M) = \text{\rm End}(T_xM)
+ \text{\rm Hom} (T_xM,T_x ^{\perp} M) + \text{\rm
Hom}(T_x^{\perp} M, T_xM) + \text{\rm End}(T_x^{\perp} M).
$$
Using (\ref{shape operator}) and calculating the curvature
operator $\widetilde R_{XY}=[\widetilde\nabla_X,
\widetilde\nabla_Y]-\widetilde\nabla_{[X,Y]}$ of the connection
$\widetilde \nabla$, we get the following \textbf{Gauss-Codazzi-Ricci
equations}:
\[
\begin{array}{llll}
\text{(Gauss)} &   R^{\top\top}_{XY} &= R_{XY} -A_XA^t_Y + A_YA^t_X &  (\top\top)   \\
         & & = R_{XY} -\sum_i \kappa_i  A^{\xi_i}X\wedge A^{\xi_i}Y & \quad\\

\text{(Codazzi 1)} & R^{\bot\top}_{XY} & = h_X \nabla_Y -h_Y
\nabla_X + \nabla_X^\perp h_Y
-\nabla_Y^\perp h_X - h_{[X,Y]}&  (\bot\top)\\
& R^{\bot\top}_{XY}Z & =(\nabla'_Xh)(Y,Z)-(\nabla'_Yh)(X,Z) \\

\text{(Codazzi 2)} & R^{\top\bot}_{XY} & = -h_X^t
\nabla_Y^\perp  -\nabla_X h_Y^t + h_Y^t \nabla_X^\perp +\nabla_Y h_X^t + h_{[X,Y]}^t & (\top\bot)\\

 & R^{\top\bot}_{XY}\eta & = -(\nabla_X A^\eta - A^{\nabla^\bot_X\eta})Y
 + (\nabla_YA^\eta-A^{\nabla^\bot_Y\eta})X \\

& &=(\nabla'_X h^t)(Y,\eta)-(\nabla'_Y h^t)(X,\eta) \\

 \text{(Ricci)} &
R^{\bot\bot}_{XY} & = R^\bot_{XY} -h_X \circ h_Y^t +h_Y
\circ  h_X^t &  (\bot\bot)\\

 & &= R^{\perp}(X,Y)
 - \sum_{a,b} \kappa_a \kappa_b <A^{\xi_a}X,A^{\xi_b}Y> \xi_a  \wedge \xi_b\\

& &= R^{\perp}(X,Y)
 - \sum_{a,b} \kappa_a \kappa_b <[A^{\xi_a},A^{\xi_b}]X,Y> \xi_a
\wedge \xi_b\\

& R^{\bot\bot}_{XY}\eta &= R^\bot_{XY}\eta-\sum_i \kappa_i \langle
X,[A^{\xi_i},A^\eta]Y\rangle \xi_i

\end{array}
\]
where $\xi_i$ is an orthonormal basis of $T^\bot M$ and  $\kappa_i=<\xi_i,\xi_i>$,  $X,Y \in TM$,
$\eta \in T^\bot M$, $R$, $R^\bot$ are the curvature tensors of
the connections $\nabla$ and $\nabla^\bot$. We identify a bivector
$X\wedge Y$ with the skew-symmetric operator $Z\mapsto \langle
Y,Z\rangle X-\langle X,Z\rangle Y$ and  denote by \mbox{$h_X^t : T^\perp M \rightarrow TM$}
the adjoint operator of \mbox{$h_X=h(X,\cdot) :TM\rightarrow T^\perp M$}.
 We recall that shape ope\-ra\-tor and second fundamental form
satisfy \mbox{$A^\eta X=h_X^t\eta$}, $\;(AX=h_X^t)$ or, equivalently,
$ <A^\eta X,Y>=<h(X,Y),\eta>$.


\begin{defi} Let $M$ be a submanifold $M$ of a
Riemannian manifold $\widetilde M$. Then
\begin{itemize}
\item[(1)] $M$ is called \textbf{curvature invariant} if
$ \widetilde R_{XY}Z \in TM \quad , \quad \forall \, X,Y,Z \in TM ,
$\\
or equivalently, $
R^{\top\bot} = R^{\bot \top} = 0.$
\item[(2)] $M$ is called \textbf{strongly curvature invariant}
if it is curvature invariant and moreover $\widetilde R_{\xi\eta}\zeta \in T^\bot M \quad , \quad \forall \,
\xi,\eta,\zeta \in T^\bot M$.
\item[(3)]
$M$ is called \textbf{parallel} if the second fundamental form is
parallel: \mbox{$\nabla'h=0$}.
\end{itemize}
\end{defi}

Let us recall the following known result.
\begin{prop}
 A parallel submanifold $M$ of a locally symmetric manifold $\widetilde M$ is
curvature invariant and locally symmetric.
\end{prop}
\begin{proof}First statement follows from ($\bot\top$).
For the second statement observe that $\widetilde R |_{T_xM}=
R^{TT} + R^{\perp T}= R^{TT}$. Then $0= \widetilde \nabla \widetilde R=
\nabla R^{TT}$. It follows
\[0=\nabla (R^{TT})(X,Y)=(\nabla R)_{XY} - \nabla (h_X^t
\circ h_Y) +\nabla(h_Y^t \circ h_X) \quad \forall X,Y \in TM \]
which implies $\nabla R=0$.
\end{proof}

\subsection{Gauss-Codazzi-Ricci  equations for a $\epsilon$-K\"ahler
submanifold}

By spe\-ci\-fying the previous formulas to a totally
$\epsilon$-complex submanifold  and using Lemma (\ref{max Kaehler proprieties})
and (\ref{derivative of second fundamental form in Kaehler
submanifolds}) we get the following

\begin{prop} \label{G.C. equations for max Kaehler submanifold in pqK}
The Gauss-Codazzi-Ricci equations for a maximal totally
$\epsilon$-complex submanifold $(M^{2n},\mathcal{J^\epsilon},g)$
of a \pqK manifold $(\widetilde M^{4n},Q,\widetilde g)$ can
be written as
\begin{itemize}
\item[(1)]  $R^{TT}_{XY} = R_{XY} + [C_X, C_Y]$
\item[(2)]  $J_2R^{\bot\bot}_{XY}J_2 = R^N_{XY} + [C_X, C_Y]
 = R_{XY} + [C_X, C_Y] -\epsilon d\omega(X,Y)\mathcal{J^\epsilon}$
\item[(3)] $J_2 R^{\bot T}_{XY} = P_{XY} - P_{YX}$
\end{itemize}
where $(J_\alpha)$ is an adapted basis of
$(M^{2n},\mathcal{J^\epsilon})$, $C=J_2h$ is the shape tensor and
$P_{XY} := (\nabla_XC)_Y +\epsilon \omega(X)\mathcal{J^\epsilon}
\circ C_Y\, \in S_{\mathcal J^\epsilon}$.

\end{prop}
\begin{proof}
We prove the first two equations since the third comes directly
from (\ref{derivative of second fundamental form in Kaehler
submanifolds})
$$
\begin{array}{ll}
R^{TT}_{XY} &=  R_{XY} -A_XA^t_Y + A_YA^t_X =   R_{XY}
+C_{J_2A_Y^t}X -C_{J_2A_X^t}Y\\
&=R_{XY} +C_{J_2h_Y}X -C_{J_2h_X}Y= R_{XY}
+C_{C_Y}X -C_{C_X}Y = R_{XY} + [C_X, C_Y]\\
\end{array}
$$
$$
\begin{array}{ll}
 J_2R^{\bot\bot}_{XY}J^{-1}_2 & = J_2(\nabla_X^\perp
J_2^2
 \nabla_Y^\perp) J_2 - J_2(\nabla_Y^\perp J_2^2
 \nabla_X^\perp) J_2 - J_2 \nabla_{[X,Y]}^\perp J_2\\
&  + J_2 h_Y  A_X J_2 -   J_2 h_X
 A_Y J_2= R_{XY}^N -J_2 h_Y C_X + J_2 h_X C_Y\\
& =R_{XY}^N -C_Y C_X +C_X C_Y=  R^N_{XY} + [C_X, C_Y]
\end{array}
$$

Now we prove that  $P_XY \in S_{\mathcal J^\epsilon}$.
Let $X,Y,Z,T \in T_xM$. From (\ref{derivative of second fundamental form in Kaehler submanifolds}), $P_{XY}= J_2(\nabla_X' h)_Y$ and computing
\[
\begin{array}{ll}
<P_{XY}  \mathcal J^\epsilon Z,T>  =  <J_2(\nabla_X' h)_Y \mathcal J^\epsilon  Z,T>\\
=<J_2  \widetilde \nabla_X [h(Y,\mathcal J^\epsilon Z),T>
-<J_2  h(\nabla_X Y,\mathcal J^\epsilon Z),T>
-<J_2  h( Y, \nabla_X (\mathcal J^\epsilon Z)),T>\\
=<J_2 J_1 \nabla^\perp_X [h(Y,Z)],T> -<J_2 J_1  [h(\nabla_X Y,Z)],T> -<J_2 J_1  [h( Y,\nabla_XZ)],T>\\
= -<P_{XY}Z,\mathcal J^\epsilon T>
\end{array}
\]
Being  $-R^{T\bot}_{XY}$ the adjoint of $R^{\bot
T}_{XY}$, the operator $R^{T\bot}_{XY}J_2$ is the adjoint of $J_2R^{\bot
T}_{XY}$.
\end{proof}

\begin{coro} \label{Ricci tensor for max Kaehler submanifold in pqK}
The Ricci tensor $\text{\rm Ric}_M$ of the $\epsilon$-K\"ahler submanifold
$M^{2n}\subset \widetilde M^{4n}$ is given by
\[
\begin{array}{l}
\text{\rm Ric}_M =\text{\rm Ric}(R^{TT}) + \text{\rm
tr}_g\langle C_{.},C_{.}\rangle = \text{\rm Ric}(R^{TT})+\langle
\sum_i \mu_i C^2_{E_i}\cdot,\cdot\rangle\\
\text{or, more precisely,}\\
\text{\rm Ric}_M(X,Y) =\text{\rm
Ric}(R^{TT})(X,Y) + \sum_{i=1}^{2n}  \mu_i \langle
C_{E_i}X,C_{E_i}Y\rangle \qquad \qquad X,Y \in TM
\end{array}
\]
where $\text{\rm Ric}(R^{TT})$ is the Ricci tensor of the
tangential part $R^{TT}$ of $\widetilde R$, that is \newline
$\text{\rm Ric}(R^{TT})(X,Y)$ $=\text{\rm tr}(Z\mapsto
R^{TT}_{ZX}Y$).
\end{coro}

\begin{proof}
\[
\begin{array}{ll}
Ric(X,Y)& =\sum_{i=1}^{2n} \mu_i \{g(R^{TT}(E_i,X)Y,E_i) -
g([C_{E_i},C_{X}]Y,E_i) \} \\ &= Ric(R^{TT})(X,Y) - \sum_{i=1} ^{2n}
\mu_i \{g(C_{E_i} C_X Y, E_i) - g(C_X C_{E_i} Y, E_i)\} \\ & =
Ric(R^{TT})(X,Y) - \sum_{i=1} ^{2n} \mu_i g(C_{E_i}E_i, C_X Y)
+\sum_{i=1} ^{2n} \mu_i g(C_{E_i} Y, C_{E_i} X)\\ &=Ric(R^{TT})(X,Y) +
\sum_{i=1} ^{2n} \mu_i g(C_{E_i} X, C_{E_i} Y)
\end{array}
\]

\end{proof}

\begin{prop} \label{CNS parallel and curvature invariant of a max Kaehler in pqK}
Let $M^{2n}$ be a $\epsilon$-K\"ahler submanifold of a \pqK
ma\-ni\-fold $\widetilde M^{4n}$. Then
\begin{itemize}
\item[(1)] $M^{2n}$ is parallel if and only if $P_{XY} :=
(\nabla_XC)_Y + \epsilon \omega(X) \mathcal{J^\epsilon}\circ
C_Y=0$ ;
\item[(2)] $M^{2n}$ is curvature
invariant if and only if the tensor $P_{XY}$ belongs to   $$
S^{(2)}_\mathcal{J^\epsilon} = \{ A \in Hom(TM,
S^{(1)}_\mathcal{J^\epsilon}) , A_{XY} = A_{YX} \}.
$$
Then $M^{2n}$ is strongly curvature invariant.
\end{itemize}
\end{prop}

\begin{proof} 1) Follows from (\ref{derivative of second fundamental form in Kaehler submanifolds}).
First statement of 2) follows from (3) of Proposition (\ref{G.C.
equations for max Kaehler submanifold in pqK}). To prove the last
statement, we use the general identity for $\widetilde R$ of $\widetilde M^{4n}$ \\
\mbox{$<\widetilde R(J^\epsilon X,J^\epsilon Y)J^\epsilon T,J^\epsilon Z> =
<\widetilde R(X,Y)T,Z>$} (it follows from repeated applications  of (\ref{action of the curvature operator})). By the
curvature invariance  and  since  $J_2T_xM=T_x^\bot M, \, \forall \, x \in M$, it is
$0=<\widetilde R(X,Y)Z, \xi>=<\widetilde R(J_2X,J_2Y)J_2 Z,J_2 \xi>, \quad X,Y,Z \in
TM, \; \xi \in T^\perp M$. Then $\widetilde R_{\xi\eta}\zeta \in T^\bot M, \quad \forall \,
\xi,\eta,\zeta \in T^\bot M$.
\end{proof}
\begin{prop} \label{second Gauss-Codazzi follows from first in max Kaehler submanifold}
 For a $\epsilon$-K\"ahler submanifold $(M^{2n},\mathcal{J^\epsilon},g)$
of a \pqK manifold $\widetilde M^{4n}$,we have:
\begin{equation} \label{normal curvature}
 R_{XY}^{\bot\bot} = J_2 R^{TT}_{XY} J_2 +
\epsilon \nu F(X,Y)J^\epsilon
\end{equation}
i.e  Ricci  equation  follows from Gauss one. Moreover
\[d\omega(X,Y) = \nu F(X,Y).\]
\end{prop}
\begin{proof}
By proposition (\ref{curvature in pqKm}), the fact that
$[J_\alpha,J_\beta]= 2 \epsilon_3 \epsilon_\gamma J_\gamma$ and from (\ref{action of the curvature operator})
one has
\[
\begin{array}{ll} \langle J_2 R^{\bot\bot}_{XY}J_2U,V\rangle &=
\langle J_2\widetilde R_{XY}J_2U,V\rangle = \langle J_2\{
[\widetilde R_{XY},J_2]U +J_2 \widetilde
R_{XY}U\},V\rangle\\
&= \langle\widetilde R_{XY}U,V\rangle + \epsilon_3 \nu  \langle
J_2 (-F_1(X,Y)J_3 +
F_3(X,Y)J_1)U,V\rangle\\
& = \langle R^{TT}_{XY}U,V\rangle - \epsilon  \nu \langle
F(X,Y)J_1U,V\rangle, \qquad X,Y,U,V \in TM,
\end{array}
\]
that is (\ref{normal curvature}). Since $J_2 R^{\bot\bot}_{XY}J_2=R^{TT}_{XY} - \epsilon
d\omega(X,Y)\mathcal{J^\epsilon}$, the last identity follows.
\end{proof}

\subsection{Maximal $\epsilon$-K\"ahler submanifolds of a para-quaternionic
symmetric space}

 Now we assume that the  manifold
$(\widetilde M^{4n}, \widetilde g)$ is a (lo\-cal\-ly) symmetric manifold, i.e.
$\widetilde \nabla\widetilde R=0$. By adapting the proof of Proposition 2.10 in \cite{AM} to the para-quaternionic case, we can state the following

\begin{prop} \label{nabla R(TT) R(perp T) R(T,perp) R(perp perp)}
Let $(M^{2n},\mathcal{J^\epsilon},g)$ be an $\epsilon$-K\"ahler
submanifold of a para-qua\-ter\-nio\-nic locally symmetric space
$(\widetilde M^{4n},Q, \widetilde g)$. Then the covariant derivatives of the
tangential part $R^{TT}$, the normal part $R^{\bot \bot}$ and
mixed part $R^{\bot T}$ of the curvature tensor $\widetilde
R_{|M}$ can be expressed in terms of these tensors and the shape
operator $C = J_2 \circ h$ as follows:

\begin{equation} \label{nabla R(TT)}
\begin{array}{ll}
\langle(\nabla_XR^{TT})(Y,Z)U, V\rangle & = + \langle R^{\bot
T}(Y,Z)U, J_2C_XV\rangle
-\langle R^{\bot T}(Y,Z)V,J_2C_XU\rangle  \\
& + \langle J_2 R^{\bot T}(U,V)C_XY,Z \rangle + \langle R^{\bot
T}(U,V)Y,J_2C_XZ\rangle
\end{array}
\end{equation}

\begin{equation} \label{nabla R(perp T)}
\begin{array}{ll}
(\nabla'_XR^{\bot T})(Y,Z)U &=
-J_2C_XR^{TT}(Y,Z)U - R^{\bot\bot}(Y,Z)J_2C_XU\\
& +[\widetilde R(J_2C_XY,Z)U + \widetilde R(Y,J_2C_XZ)U]^\bot \\
&= -J_2C_XR^{TT}(Y,Z)U - J_2R^{TT}(Y,Z)C_XU \\ & +\nu \epsilon
F(Y,Z)J_3C_XU  +[\widetilde R(J_2C_XY,Z)U + \widetilde
R(Y,J_2C_XZ)U]^\bot
\end{array}
\end{equation}

\begin{equation} \label{nabla R(T perp)}
\begin{array}{ll}
(\nabla_X'R^{T\bot})(Y,Z)\xi = & +\big[\widetilde R(J_2C_XY,Z)\xi
+
\widetilde R(Y,J_2C_XZ)\xi\big]^T \\
& + R^{TT}(Y,Z)C_{J_2\xi}X - C_XJ_2R^{\bot\bot}(Y,Z)\xi
\end{array}
\end{equation}

\begin{equation} \label{nabla R(perp perp)}
\begin{array}{ll}
\langle (\nabla_X'R^{\bot\bot})(Y,Z)J_2U,J_2V\rangle  = & \langle
R^{\bot T}(U,V)Z,J_2C_XY\rangle
- \langle R^{\bot T}(U,V)Y,J_2C_XZ\rangle \\
& + \langle R^{\bot T}(Y,Z)C_UX,J_2V\rangle + \langle
C_XR^{T\bot}(Y,Z)J_2U,V\rangle
\end{array}
\end{equation}

for any $X,Y,Z,U,V \in TM \, , \xi \in T^\bot M$.

\end{prop}

 By (\ref{nabla R(TT)}) and (\ref{nabla R(perp T)})  we get immediately the following result.

\begin{prop} \label{max Kaehler curvature invariante then R(TT) parallel}
If the $\epsilon$-K\"ahler submanifold $M^{2n}\subset\widetilde
M^{4n}$ is curvature invariant then the
tensor field $R^{TT}$ is parallel i.e. $ \nabla R^{TT} =0$
and satisfies the identity
\begin{equation} \label{proprieties of max Kaehler submanifold which is curvature invariant}
\begin{array}{l}
C_XR^{TT}(Y,Z) + R^{TT}(Y,Z)C_X +  \nu  \epsilon F(Y,Z)J^\epsilon C_X \qquad \qquad \qquad\\

\qquad \qquad  \qquad \qquad \qquad =\big[J_2(\widetilde R(J_2C_XY,Z)+\widetilde
R(Y,J_2C_XZ))\big]^{TT}
\end{array}
\end{equation}
where $(A)^{TT}$ denotes the $\text{\rm End}(T_xM)$ component of
an endomorphism $A$ of $T_x \widetilde{M}$.
\end{prop}

Denote by $[C,C]$ the $\text{End}(T_xM)-$valued 2-form, given by
$$[C,C](X,Y) = [C_X,C_Y] \qquad \qquad \forall \, X,Y \in TM.$$
(One can easily check that it is globally defined on $M$).
\par
For a subspace ${\mathcal G} \subset \text{End}(T_xM)$ we define
the space ${\mathcal R}({\mathcal G})$ of the {\it curvature
tensors of type} $\mathcal G$ by
$$
{\mathcal R}({\mathcal G})=\lbrace R \in {\mathcal
G}\otimes{\Lambda^2}T^*_xM \quad | \quad\text{\rm
cycl} \;R(X,Y)Z=0,\, \forall \, X,Y,Z \in T_xM\rbrace$$ where
 cycl  is the sum of cyclic permutations of $X,Y,Z$.

Let denote by ${\mathfrak u}^\epsilon_{p,q}$ the Lie algebra of the unitary Lie group of automorphisms of the
Hermitian (para)-complex  structure $(\mathcal{J}^\epsilon,  g)$ where $(p,q)$ corresponds to the signature  of $g$.
As a Corollary of Propositions (\ref{nabla R(TT) R(perp T)
R(T,perp) R(perp perp)})  and  (\ref{G.C. equations for
max Kaehler submanifold in pqK}) (1) we have the following
 \begin{prop} Under the assumptions of Proposition
 (\ref{max Kaehler curvature invariante then R(TT) parallel})
the tensor field  $\quad [C,C]=R^{TT}-R$ belongs to the
space ${\mathcal R}({\mathfrak u}^\epsilon_{p,q})$ and satisfies the second
Bianchi identities:
$$
\text{\rm cycl} \,\nabla_Z[C_X,C_Y]=0.
$$
\end{prop}
 \begin{proof} The tensor $[C,C]$ satisfies the first
Bianchi identity since $R$ and $R^{TT}$ do it. Moreover
$\mathcal{J^\epsilon} \circ [C_X,C_Y]= -C_X \circ
\mathcal{J^\epsilon} C_Y + C_Y \circ \mathcal{J^\epsilon} C_X= C_X
C_Y \circ \mathcal{J^\epsilon} - C_Y  C_X \circ
\mathcal{J^\epsilon} = [C_X,C_Y] \circ \mathcal{J^\epsilon}$ i.e.
$[C_X,C_Y]$ commutes with $\mathcal{J^\epsilon}$. Furthermore, by
the symmetry of $C_X$, $<[C_X,C_Y]Z,T>= <C_Y Z, C_X T> - < C_X Z,
C_Y T>= <Z,[C_Y,C_X]T>$ that is $[C_X,C_Y]$  is skew-symmetric
with respect to the metric $g= < \; , \; > $. Then the tensor
$[C,C]$ belongs to the space ${\mathcal R}({\mathfrak u}^\epsilon_{p,q})$ of the
${\mathfrak u}^\epsilon_{p,q}$-curvature tensors. The last statement follows from
remark that $\text{\rm cycl}\,\nabla_Z[C_X,C_Y]= \text{\rm
cycl}\,\nabla_Z(R^{TT} + R)$. But $\nabla R^{TT}=0$ and $R$
satisfies the second Bianchi identity.
\end{proof}
As another Corollary of Proposition (\ref{nabla R(TT) R(perp T)
R(T,perp) R(perp perp)}) we get the following result.
\begin{prop} \label{CNS max Kaehler in loc.sym. pqK to be loc. sym.}
A maximal $\epsilon$-K\"ahler submanifold $M^{2n}$ of a locally
symmetric \pqK manifold $\widetilde M^{4n}$ is locally symmetric
(that is $\nabla R = 0$) if and only if the following identity
holds:
\begin{equation} \label{CNS max Kaehler in locally symmetric pqK, to be locally symmetric}
\begin{array}{ll}
\langle \nabla_X [C,C]_{Y,Z} U,V\rangle
&=  \langle R^{\bot T}(Y,Z)U, J_2C_XV\rangle - \langle R^{\bot T}(Y,Z)V,J_2C_XU\rangle\\
& \quad +\langle J_2 R^{\bot T}(U,V)C_XY,Z \rangle + \langle R^{\bot
T}(U,V)Y,J_2C_XZ\rangle\, .
\end{array}
\end{equation}

If  $M$ is curvature invariant then (\ref{CNS max
Kaehler in locally symmetric pqK, to be locally symmetric})
 reduces to the
condition that the tensor field $[C,C]$ is parallel
($\nabla[C,C]=0$).
\end{prop}
\begin{proof} The proof follows directly from the Gauss equation
and (\ref{nabla R(TT)}).
\end{proof}

\subsection{Maximal totally complex submanifolds of para-quaternionic space forms}

Now  we assume that $(\widetilde M^{4n},Q,\widetilde g)$ is
a non flat para-quaternionic space form, i.e. a \pqK manifold
which is locally isometric to the para-quaternionic projective
space $\widetilde{\HH} P^n$ or the dual para-quaternionic
hyperbolic space $\widetilde{\HH} H^n$ with standard metric of
reduced scalar curvature $\nu$. Recall
that the curvature tensor
of $(\widetilde M^{4n},Q,\widetilde g)$ is given by
$\widetilde R=\nu R_0$  (see (\ref{curvature tensor of the para-quaternionic projective space})).
  We denote by $R_{\mathbb \C P^n}^\epsilon$ the curvature tensor  of the $\epsilon$-complex projective
space  (nor\-ma\-lized such that the holomorphic curvature equal to
1):
$$
R_{\mathbb \C P^n}^\epsilon(X,Y) = \frac{1}{4}\Big(-\epsilon X \wedge
Y + JX\wedge JY -2\langle JX, Y\rangle J\Big) \, .
$$
It is a straightforward to verify the
 \begin{prop} \label{R(TT) R(perp T) R(T perp) R(perp perp) in max Kaehler of pqK space form}
 Let $(M^{2n},\mathcal{J^\epsilon},g)$ be a totally $\epsilon$-complex submanifold  of the para-quaternionic
   space form $\widetilde M^{4n}$.  We
have:
\begin{itemize}
\item[(1)]
$\qquad R^{TT}_{XY} = -\epsilon \nu(R_{\mathbb \C P^n}^\epsilon)_{XY}
=\epsilon \frac{\nu}{4}\Big(\epsilon X\wedge Y - J_1 X\wedge J_1 Y
+ 2\langle J_1 X, Y\rangle J_1\Big)$.
\item[(2)] \qquad $\text{\rm Ric}(R^{TT})=\frac{\nu}{2}(n+1)g \, , \, g=\widetilde
g_{|M}$.\\
\item[(3)]
$\qquad R^{\bot T} = R^{T\bot} =0$.
\item[(4)]
$\qquad R^{\bot \bot}_{XY} = \frac{\nu}{4}\Big(-J_2 X\wedge J_2 Y
+ \epsilon J_3 X\wedge J_3 Y + 2 \epsilon \langle J_1X,Y\rangle
J_1\Big)$.
\end{itemize}
\end{prop}

As a consequence of Corollary (\ref{Ricci tensor for max Kaehler
submanifold in pqK}) and Proposition (\ref{R(TT) R(perp T) R(T
perp) R(perp perp) in max Kaehler of pqK space form}) we get
\begin{prop} \label{Ricci curvature of a Kaehler submanifold of a paraquaternionic space form}
Let $M^{2n}$ be a $\epsilon$-K\"ahler
submanifold  of a para-quaternionic space form $\widetilde M^{4n}$
with reduced scalar curvature $\nu$.
\[
\begin{array}{lll}
 \text{\rm Ric}_M(X,X)&=&\frac{\nu}{2}(n+1)g(X,X) + \text{\rm
 tr}C_X^2\\
&=&\frac{\nu}{2}(n+1)||X||^2 - \sum_{i=1}^{2n} \mu_i ||h(E_i,X)||^2, \qquad X \in T_xM
\end{array}
\]

 Moreover the second fundamental
form $h_x$ of $M$ at point $x\in M$ vanishes if and only if
$(\text{\rm Ric}_M)_x=\frac{\nu}{2}(n+1)g$. In particular $M$ is a
totally $\epsilon$-complex totally geodesic submanifold if and
only if
$$\text{\rm Ric}_M=\frac{\nu}{2}(n+1)g \, .$$
\end{prop}

From Proposition (\ref{CNS max Kaehler in loc.sym. pqK to be loc.
sym.}) we get
\begin{prop} \label{parallel max Kaehler in space form}
A maximal $\epsilon$-K\"ahler submanifold
$(M^{2m},\mathcal{J^\epsilon},g)$ of a non flat para-quaternionic
space form is locally symmetric if and only if the tensor field
$[C,C]$ is parallel. In particular, any maximal
$\epsilon$-K\"ahler submanifold with parallel second fundamental
form is (locally ) symmetric.
\end{prop}
\begin{proof} It is sufficient to prove only the last
statement. Assume that $\nabla'h=0$. Then $\nabla_XC= \epsilon
\omega(X)\mathcal{J^\epsilon} C \quad$ and $\quad
\nabla_X[C,C](Y,Z)=[\nabla_XC_Y,C_Z]+[C_Y,\nabla_XC_Z]=$\\
$ \epsilon
\omega(X)\Big([\mathcal{J^\epsilon}C_Y,C_Z]+[C_Y,\mathcal{J^\epsilon}C_Z]\Big)=0 \quad
$  since $C_Y$ anticommutes with $\mathcal{J}^\epsilon$.

\end{proof}

\section{The parallel cubic line bundles of a maximal parallel $\epsilon$-K\"ahler submanifold}

For a deep analysis of parallel submanifolds of a quaternionic manifold refer to \cite{AM}, \cite{20}.
We will assume that $\widetilde M^{4n}$ is a \pqK manifold with
the reduced scalar curvature $\nu\neq 0$. We consider first the
case  that $(M^{2n},J,g)$ is a   parallel totally complex
submanifold of $\widetilde M$. From Proposition (\ref{CNS parallel and curvature invariant of a max Kaehler in pqK})
\[P_{XY}:=(\nabla_X C)_Y - \omega(X)J\circ C_Y=0 \qquad X,Y\in TM.
\]
We will assume moreover that $M$ is not a totally geodesic submanifold, i.e. $h \neq 0$.
By Proposition  (\ref{CNS parallel and curvature invariant of a
max Kaehler in pqK}) $M$ is a curvature invariant submanifold
($R^{\bot T}=0$). We denote by  $T^{\mathbb C}M=T^{1,0}M+T^{0,1}M$
the decomposition of the complexified tangent bundle into
holomorphic and antiholomorphic parts and by $T^{*\mathbb
C}M=T^{*1,0}M+T^{*0,1}M$ the dual decomposition of the cotangent
bundle. \par Denote by $S^{(1)\mathbb C}_J$ the complexification
of the bundle $S^{(1)}_J$ (see Corollary (\ref{derivative of C}))
and by $g\circ S_J^{(1)\mathbb C}$ the associated subbundle of the
bundle $S^3(T^*M)^{\mathbb C}$. We will call ${S^3(T^*M)}^{\mathbb
C}$ the \textbf{bundle of complex cubic forms}.

\begin{prop} \label{parallel max Kaehler has got a cubic line bundle}
Let $(M^{2n},J,g)$ be a parallel K\"ahler submanifold of a
para-qua\-ter\-nio\-nic K\"ahler manifold $\widetilde M^{4n}$ with $\nu
\neq 0$.  If it is not totally geodesic then on $M$ there is a pair of
canonically defined parallel complex line subbundles $L$ (resp. $\overline{L}$) of the
bundle $S^3(T^{*1,0}M)$  (resp. $S^3(T^{*0,1}M)$) of holomorphic (resp. antiholomorphic) cubic forms such that the
curvature induced by the Levi-Civita
connection  has the curvature form
\begin{equation} \label{curvature of the parallel subbundle L in Kaehler manifolds}
R^L=-i\nu F, \qquad (\text{resp.} \; R^{L'}=i\nu F).
\end{equation}
\end{prop}
\begin{proof} We first prove the
following
\begin{lemma} \label{decomposition of cubic form}
$g\circ S^{(1)\mathbb C}_J = S^3(T^{*1,0}M)+S^3(T^{*0,1}M)$.
\end{lemma}

\begin{proof} Since $J|_{T^{1,0}M}=i \, Id, \; J|_{T^{0,1}M}=-i\, Id$, one has
$$
S^{\mathbb C}_J = \text{Hom}(T^{1,0}M,T^{0,1}M)+\text{Hom}(T^{0,1}M,T^{1,0}M)
$$
where $S^{\mathbb C}_J $ is  the space $S^{\mathbb C}_J $ of
complex en\-do\-mor\-phisms of $T_x^{\mathbb C}M$ which anticommute with $J$.
In fact, let  $X \in T^{1,0}M$, $\; A \in S^{\mathbb C}_J $, and
denote by $AX=Y=Y^{1,0}+ Y^{0,1}, \; Y^{1,0} \in T^{1,0} M,  \;
Y^{0,1} \in T^{0,1} M$. Then $AJX=iAX=iY^{1,0} +iY^{0,1}$ whereas
$-JAX=-JY^{1,0} - JY^{0,1} =-iY^{1,0} +iY^{0,1}$ which implies
$Y^{1,0}=0$. Analogously, if $X \in T^{0,1}M$, we get $Y^{0,1}=0$.

Hence the space $ g\circ S^{\mathbb C}_J $ of symmetric bilinear
forms, associated with $S^ {\mathbb C}_J$ is
$$
g\circ S^{\mathbb C}_J = S^2(T^{*1,0}M)+S^2(T^{*0,1}M)\, .
$$
In fact, for $X \in T^{1,0}M$ and $A \in  S^{\mathbb C}_J$,
\[<AX,Y>= -<J^2AX,Y>=<JAX,JY>=-<AJX,JY>=
-i<AX,JY>.\]
This implies that $JY=iY$ i.e $Y \in T^{1,0}M$.  Analogously, if
$X \in T^{0,1}M$ then $Y \in T^{0,1}M$. This proves the lemma.
\end{proof}
Using this lemma we can decompose the cubic form $gC \in
g\circ S^{(1)}_J$ associated with the shape operator $C=J_2h$ into
holomorphic and antiholomorphic parts:
$$
gC=q+{\overline q} \in S^3(T^{*1,0}M) + S^3(T^{*0,1}M) \, .
$$
Since, by assumption, $\nabla_XC = \omega(X)J\circ C$ we have
$$ g \nabla_XC =\nabla_X (gC) = \nabla_Xq + \nabla_X \bar q =
\omega(X)g(J \circ C).$$
 For $ Y,Z  \in T^{1,0}M $, we get
$$ \nabla_X(g C)(Y,Z) =\omega(X)g(JC(Y,Z))
=-i \omega(X)gC(Y,Z)
$$
since $ C(Y,Z) \in T^{0,1}M $ and $JC(Y,Z)= -iC(Y,Z) $. This shows
that
\begin{equation} \label{derivative of q (olomorphic part of gC)}
\nabla_Xq =-i\omega(X)q.
\end{equation}
Under the changing of adapted basis $(J_\alpha)\rightarrow
(J'_\alpha)$ represented in basis $(J_1,J_2,J_3)$ by the first matrix in
(\ref{matrix of change of adapted basis on Kaehler submanifold }),
we have $J'_2=\cos \theta J_2 + \sin\theta J_3$ and the cubic form $q$ changes by
$q \rightarrow {q'} = (\cos\theta - i \sin\theta)q \,$.
In fact, for $Y,Z \in T^{1,0}M$,
$$
q'(Y,Z) = \cos\theta J_2 \circ h(Y,Z)  + \sin\theta J_1 (J_2 \circ
h(Y,Z)) = \cos\theta q - i \sin\theta  q
$$
since $J_2 \circ h(Y,Z) \in T^{0,1}M$. Analogously $\overline{q} \rightarrow {\overline{q}'} = (\cos\theta + i \sin\theta) \overline{q}$.

Note also that the cubic forms $q$ and $\overline{q}$ are not $0$ at any point, since by
assumption the second fundamental form $h$ is parallel and not
zero. These show that the com\-plex line bundle
$L=\text{span}_{\mathbb C}(q) \subset S^3(T^{*1,0}M)$ (resp. $\overline{L}=\text{span}_{\mathbb C}(\overline{q}) \subset S^3(T^{*0,1}M)$)
 is globally
defined and parallel, i.e. the Levi-Civita  connection $\nabla$
preserves $L$ (resp. $\overline{L}$) and defines a connection $\nabla^L$ in $L$ (resp. $\nabla^{\overline{L}}$ in $\overline{L}$).  Using
(\ref{derivative of q (olomorphic part of gC)}), we calculate the
curvature of $\nabla^L$ as follows:
$$
\begin{array}{ll}
R^L(X,Y)q &= \Big([\nabla^L_X,\nabla^L_Y]-\nabla^L_{[X,Y]}\Big)q = \Big([\nabla_X,\nabla_Y]-\nabla_{[X,Y]}\Big)q \\
& = -\nabla_X\Big( \omega(Y)iq\Big) + \nabla_Y\Big(
\omega(X)iq\Big)
+\omega([X,Y])iq \\


&= -d\omega(X,Y)iq  -\omega(Y) \omega(X) q +  \omega(X) \omega(Y)
q\\
& = -d\omega(X,Y)iq= -\nu F(X,Y)iq
\end{array}
$$
Analogously it is  $\nabla_X {\overline q} =i\omega(X) {\overline q}$ and
$R^{\overline{L}}(X,Y) {\overline q}= \nu F(X,Y)i{\overline q} \,$.
\end{proof}

\begin{defi} A parallel subbundle $L\subset
S^3(T^{*1,0}M)$ with the curvature form (\ref{curvature of the
parallel subbundle L in Kaehler manifolds})  on a K\"ahler
manifold $M$ is called a \textbf{parallel cubic line bundle of
type} $-\nu$.
\end{defi}

\begin{coro} \label{parallel max Kaehler has got a cubic line bundle of type nu}
A parallel maximal K\"ahler not totally geodesic submanifold $M$
of a \pqK manifold $\widetilde M$ with $\nu \neq 0$ has a pair of parallel
cubic line bundles of type $\pm \nu$.
\end{coro}

Let  consider now the case that $(M^{2n},K,g)$ is a parallel,
totally para-complex, not totally geodesic  submanifold of
$\widetilde M$.  Then  \[P_{XY}=(\nabla_X C)_Y + \omega(X)K \circ C_Y=0 \quad X,Y\in TM \qquad  \text{ and} \qquad  h \not \equiv 0. \]

By Proposition  (\ref{CNS parallel and curvature invariant of a
max Kaehler in pqK}) $M$ is a curvature invariant submanifold.
Let $TM=T^{+}M+T^{-}M$ be the bi-Lagrangean decomposition of
the  tangent bundle into the $(+1)$ and $(-1)$ eigenspaces of $K$
 and by $T^*M
=(T^{*+}M)+(T^{*-}M)$ the dual decomposition of the cotangent
bundle. We will call ${S^3(T^*M)}$ the \textbf{bundle of real
cubic forms}. \par We recall that  $C \in S^{(1)}_K$ (see
Corollary (\ref{derivative of C})) and we denote by $g\circ
S_K^{(1)}$ the associated subbundle of the bundle $S^3(T^*M)$.
Following the same line of proof of lemma (\ref{decomposition of
cubic form}) we can affirm that
\begin{lemma} $g\circ S^{(1)}_K =
S^3(T^{*+}M)+S^3(T^{*-}M)$.
\end{lemma}
 We can then decompose the cubic form $gC \in g\circ
S^{(1)}_K$ associated with the shape operator $C=J_2 h$ according
to :
$$
gC=q^{+} + q^{-} \in S^3(T^{*+}M) + S^3(T^{*-}M) \, .
$$

\begin{prop} \label{parallel max para-Kaehler has got a cubic line bundle 2}
 Let $(M^{2n},K)$ be a parallel para-K\"ahler submanifold of a para-quaternionic K\"ahler manifold
$\widetilde M^{4n}$ with $\nu \neq 0$.  If it is not totally
geodesic then on $M$ the pair of  real line subbundle $L^+:=\R q^+
\subset S^3(T^{*+}M)$ and $L^-:=\R q^- \subset S^3(T^{*-}M)$ are
globally defined and parallel, i.e  the Levi-Civita connection
$\nabla$ pre\-ser\-ves $L^+$ (resp. $L^-$)  and defines a connection
$\nabla^{L^+}$ on $L^+$  (resp. $\nabla^{L^-}$ on $L^-$) whose
curvature is
\begin{equation}
\begin{array}{l}
\label{curvature of the parallel subbunble P in para-Kaehler}
R^{L^+}= \nu F, \quad (\text {resp. } \quad  R^{L^-}= -\nu F)
\end{array}
\end{equation}
 where $F = g\circ K$ is the K\"ahler form of $M$.
\end{prop}

\begin{proof}
Following the same line of proof of the previous K\"ahler case, we have
\begin{equation} \label{derivative of q^+ and q^-}
\nabla_X q^{+} =\omega(X) q^{+}; \qquad  \nabla_X q^{-}
=-\omega(X) q^{-}.
\end{equation}


Under a changing of the adapted basis $(J_\alpha)\rightarrow
(J'_\alpha)$, represented in basis $(J_1,J_2,J_3)$ by the second matrix in
(\ref{matrix of change of adapted basis on Kaehler submanifold }),
we have that
\[
q^{+} \rightarrow q'^{+} = (\cosh \theta - \sinh \theta)q^{+}
\qquad \text{and} \qquad  q^{-} \rightarrow {q'^{-}} = (\cosh \theta +\sinh \theta) q^{-}.
\]

Note that  $q^+\neq 0$ and $q^-\neq 0$ at any point, since by
assumption the second fundamental form $h$ is parallel and not
zero and the metric $g$ is non degenerate on $M$. Then the real
line  bundles $L^+:=\R q^+ \subset S^3(T^{*+}M)$ and $L^-:=\R q^-
\subset S^3(T^{*-}M)$ are globally defined and parallel, i.e  the
Levi-Civita connection $\nabla$ preserves $L^+$ (resp. $L^-$)  and
defines a connection $\nabla^{L^+}$ on $L^+$  (resp.
$\nabla^{L^-}$ on $L^-$). Moreover
$$
\begin{array}{ll}
R^{L^+} (X,Y)q^+ &=  \Big([\nabla^{L^+}_X,\nabla^{L^+}_Y]-\nabla^{L^+}_{[X,Y]}\Big)q^+
= \Big([\nabla_X,\nabla_Y]-\nabla_{[X,Y]}\Big)q^+\\

&= \nabla_X\Big( \omega(Y)(q^{+})\Big) - \nabla_Y\Big(
\omega(X)(q^{+}) \Big)
-\omega([X,Y])(q^{+}) \\

&= d\omega(X,Y)(q^{+} )= \nu F(X,Y)q^{+}\, .
\end{array}
$$
Analogously $R^{L^-} (X,Y)q^-=-\nu F(X,Y)q^{-}$.
\end{proof}

\begin{Acknow}
I owe my deepest gratitude to   Professor Dmitri  Alekseevsky whose  precious  guidance have been fundamental to accomplish  this research.
\end{Acknow}




\begin{thebibliography}{9999}


\bibitem{A}
D.V. Alekseevsky: {\em Compact quaternion spaces}, Funct. Anal. Appl. \textbf{2} (1968), 106-114.


\bibitem{AC1}
D.V. Alekseevsky, V. Cortes: {\em The twistor spaces of a
para-quaternionic K\"ahler manifold}, Osaka J.Math. (1)  \textbf{45} (2008),  215-251.


\bibitem{AC2}
D.V. Alekseevsky, V. Cortes: {\em Classification of
pseudo-Riemannian symmetric spaces of quaternionic K\"ahler type},
Amer. Math. Soc. Transl. (2) \textbf{213} (2005),   33-62.




\bibitem{AMT}
 D. V. Alekseevsky, C. Medori, A. Tomassini: {\em "Homogeneous para-K\"ahler Einstein manifolds},
Russian Math. Surveys (1) \textbf{64}  (2009), 1-43


\bibitem{AM}
D.V. Alekseevsky, S. Marchiafava: {\em Hermitian and K\"ahler
submanifolds of a quaternionic K\"ahler manifold}, Osaka J. Math. (4) \textbf{38} (2001), 869-904.





\bibitem{25}
C. Bejan: {\em A classification of the almost para Hermitian
manifolds}, Proc. Conference on Diff. Geom. and Appl., Dubrovnik,
(1988), 23-27.


\bibitem{30}
V. Cruceanu,  P. Fortuny, P.M. Gadea: {\em A survey on para
complex geometry}, Rocky Mountain J.Math \textbf{26} (1996), 83-115.



\bibitem{DJS}  A.S. Dancer, H.R. Jorgensen, A.F. Swann  {\em Metric geometris over the split quaternions},
Rend. Sem. Univ. Pol. Torino (2) \textbf{63} (2005), 119-139.


\bibitem{LD} L.David:  {\em    About the geometry of almost para-quaternionic manifold},
Differential Geom.Appl. (5) \textbf{27} (2009), 575-588.




\bibitem{7} S. Funabashi: {\em    Totally complex submanifolds of a
quaternionic K\"ahlerian manifold},  Kodai Math. J.  \textbf{2} (1979),  314-336.






\bibitem{28}
P.M. Gadea, J. Mu\~{n}oz Masqu\'{e}: {\em  Classification of almost para
Hermitian manifolds}, Rend. Mat. Appl. \textbf{11} (1991), 337-396.


\bibitem{G}
 A.Gray:  {\em A note on manifolds whose holonomy group is a subgroup of $Sp(n) \cdot Sp(1)$}, Michigan Math. J. \textbf{16} (1969),  125-128.


\bibitem{24}
A. Gray, L.M. Hervella: {\em The sixteen classes of almost
Hermitian manifolds and their linear invariants}, Ann. Mat. Pura
Appl. \textbf{123} (1980), 35-58.


\bibitem{IMV} S. Ianus, S. Marchiafava, G.E.Vilcu: {\em    Paraquaternionic CR-submanifolds of paraquaternionic
K\"ahler manifold and semi-Riemannian submersions},  Cent. Eur. J. Math.
(4) \textbf{8} (2010), 735-753.



\bibitem{39}
S. Ivanov, S. Zamkovay: {\em Para-Hermitian and Para-Quaternionic
manifolds}, Diff.Geom. Appl. \textbf{23} (2005), 205-234.


\bibitem{M}
S. Marchiafava: {\em    Submanifolds of (para-)quaternionic
K¨ahler manifolds}, Note  Mat. \textbf{28} (2008), suppl. 1, 295-316 (2009).



\bibitem{11} A. Mart\'inez: {\em    Totally complex submanifolds of
quaternionic projective space},  Geometry and topology of
submanifolds (Marseille, 1987)  World Sci. Publishing,
Teaneck, NJ, (1989), 157-164.


\bibitem{19} M. Takeuchi: {\em    Totally complex submanifolds of
quaternionic symmetric spaces}, Japan. J. Math. (N.S.)\textbf{12}
(1986), 161-189.


\bibitem{20} K. Tsukada:  {\em    Parallel submanifolds in a
quaternion projective space}, Osaka J. Math. \textbf{22} (1985), 187-241.


\bibitem{MV} M.Vaccaro: {\em    Subspaces of a para-quaternionic
Hermitian vector space},  Int. J. Geom. Methods Mod. Phys. (7) \textbf{8}  (2011), to appear.
arXiv:1011.2947v1 [math.DG].



\bibitem{MV1}
M.Vaccaro: {\em Basics of linear para-quaternionic geometry I:
Hermitian para-type structure on a real vector space},
Bull. Soc. Sci. Lettres  {\L}\'{o}d\'{z} S{\'{e}}r. Rech. D{\'{e}}form (1)
\textbf{61} (2011), to appear.


\bibitem{MV2}
M.Vaccaro: {\em Basics of linear para-quaternionic geometry II:
Decomposition of a generic subspace of a para-quaternionic
Hermitian vector space}, Bull. Soc. Sci. Lettres  {\L}\'{o}d\'{z} S{\'{e}}r. Rech. D{\'{e}}form (2)
\textbf{61} (2011), to appear.

\end{thebibliography}
\end{document}